\documentclass[11pt]{amsart}
\usepackage{graphicx, amssymb, epstopdf, amsmath, amssymb}
\DeclareMathAlphabet{\mathcal}{OMS}{cmsy}{m}{n}

\usepackage{mathptmx} 
\usepackage[scaled=0.90]{helvet} 
\usepackage{courier} 
\normalfont
\usepackage[T1]{fontenc}

\theoremstyle{plain}
\newtheorem{Thm}{Theorem}

\newtheorem{Lem}[Thm]{Lemma}
\newtheorem{Cor}[Thm]{Corollary}
\newtheorem{Prop}[Thm]{Proposition}

\theoremstyle{definition}

\newtheorem*{Rem}{Remark}

\theoremstyle{remark}

\def\F{\mathbb{F}}
\def\cp{\hbox{${\mathbb{CP}}^2$}}
\def\cpb{\hbox{$\overline{{\mathbb{CP}}^2}$}}

\def\osigma{\overline{\sigma}}

\def\mod{\mathrm{Mod}}

\title{Branched covers and pencils on hyperelliptic Lefschetz fibrations}

\author{Terry Fuller}
\address{Department of Mathematics,
California State University, Northridge, 
Northridge, CA 91330 } 
\email{terry.fuller@csun.edu}

\begin{document}
\begin{abstract}
For every fixed $h \geq 1$, we construct an infinite family of simply connected symplectic 4-manifolds $X'_{g,h}[i]$, for all $g>h$ and $0\leq i < 2p-1$, where $p=\lfloor \frac{g+1}{h+1} \rfloor$. Each manifold $X'_{g,h}[i]$ is the total space of a symplectic genus $g$ Lefschetz pencil constructed by an explicit monodromy factorization. We then show that each $X'_{g,h}[i]$ is diffeomorphic to a complex surface that is a fiber sum formed from two standard examples of hyperelliptic genus $h$ Lefschetz fibrations, here denoted $Z_h$ and $H_h$. Consequently, we see that $Z_h, H_h,$ and all fiber sums of them admit an infinite family of explicitly described Lefschetz pencils, which we observe are different from families formed by the degree doubling procedure. \end{abstract}
\maketitle
\section{Introduction}

This paper is a sequel to \cite{BHM:Unchaining} and \cite{F:Unchaining}. In the former work,  R. \.{I}nan\c{c} Baykur, Kenta Hayano, and Naoyuki Monden introduced a family of symplectic 4-manifolds $X_g'[i]$ for all $g\geq 3$ and $0\leq i \leq g-1$, each of which admits a genus $g$ Lefschetz pencil with $2(i+1)$ base points. In \cite{BHM:Unchaining} they showed that this family has many interesting properties, among them potential exotic symplectic Calabi-Yau manifolds. (Here, ``exotic'' means not diffeomorphic to the standard K3 surface.) In \cite{F:Unchaining}, the author proved that in fact all of their examples are diffeomorphic to elliptic surfaces, and in particular that the symplectic Calabi-Yau manifolds in their list are standard. Although the Baykur-Hayano-Monden examples failed to yield new manifolds, the result in \cite{F:Unchaining} revealed a rich collection of pencil structures on elliptic surfaces, described in an explicit way, as the manifolds  $X_g'[i]$ were defined concretely through monodromy factorizations. 

In the study of Lefschetz fibrations and pencils, hyperelliptic ones have been a particular object of focus, in part because they can be viewed as generalizations of genus 1 and 2 fibrations, where classification results exist. (See \cite{Moishezon}, \cite{Chakiris}, \cite{Smith_Hodge}, \cite{ST_Annals}.) It is known that every hyperelliptic Lefschetz fibration is a 2-fold branched cover of a rational surface (\cite{Hyperelliptic}, \cite{ST_CCM}), allowing methods that will be at the core of the present article. Bernd Siebert and Gang Tian conjecture that every hyperelliptic Lefschetz fibration over $S^2$ without reducible fibers is holomorphic. They have settled this in fiber genus 2 for fibrations with transitive monodromy (\cite{ST_Annals}), and in this case it follows that every such fibration is equivalent to a fiber sum of two particular examples. The generalizations of these specific examples to arbitrary genus are studied in this paper. 

These examples, to be denoted $Z_h$ and $H_h$, arise from the well-known odd chain and hyperelliptic relations in the mapping class group $\mod(\Sigma_h)$ of a closed genus $h$ surface. In Section~\ref{sec:hyperelliptic} we define these manifolds, and review their topology; for $H_h$, we also derive an equivalent monodromy factorization. Next, in Section~\ref{sec:manifolds}, we generalize the mapping class group factorizations of Baykur, Hayano, and Monden from \cite{BHM:Unchaining}. By performing unchaining surgeries on this factorization we construct for all fixed $h \geq 1$ a family $X_{g,h}'[i]$ of symplectic Lefschetz pencils, for all $g>h$ and $0 \leq i \leq 2p-1$, where $p=\lfloor \frac{g+1}{h+1} \rfloor$. (The original Baykur-Hayano-Monden examples correspond to $h=1$.) Our main theorem appears in Section~\ref{sec:theorem}, where we generalize the diffeomorphisms constructed in \cite{F:Unchaining} to show that the total space of all of these pencils are diffeomorphic to complex surfaces obtained as fiber sums of the hyperelliptic examples $Z_h$ and $H_h$. (The number of summands are determined by the divisor and remainder when $2g+2$ is divided by $2h+2$.) As was the case in \cite{F:Unchaining}, the proof relies on using Lefschetz fibration and pencil structures to depict the manifolds as 2-fold branched covers of rational surfaces; the diffeomorphism is described entirely in the base of these covers, using traditional Kirby calculus enhanced with moves that modify the branch set by isotopy. As a consequence, $Z_h, H_h,$ and all fiber sums of them an admit infinite family of explicitly described Lefschetz pencils. In the final section, we make some remarks comparing these families to those produced by degree doubling, and speculate on possible applications to open questions in smooth and symplectic 4-manifold topology.

We assume the reader is familiar with the basic topology of Lefschetz fibrations and pencils on symplectic 4-manifolds. The preeminent reference for this topic is \cite{GompfStipsicz}. Careful summaries of background most relevant here also appear in the introductory sections of \cite{BHM:Unchaining} and \cite{F:Unchaining}.

We denote a genus $g$ surface with $m$ marked points and $n$ boundary components by $\Sigma_{g,m}^n$. The mapping class group of  $\Sigma_{g,m}^n$ will be denoted as $\mod(\Sigma_{g,m}^n)$. If $m$ or $n$ are omitted, they are assumed to be zero. 

\section{Two standard hyperelliptic Lefschetz fibrations} \label{sec:hyperelliptic}

In this section, we survey the basic properties of two families of hyperelliptic Lefschetz fibrations. The curves $c_i$ are indicated in Figure~\ref{fig:Sigma_h_c}.
\begin{figure}[ht]
\begin{center}
\includegraphics[scale=1]{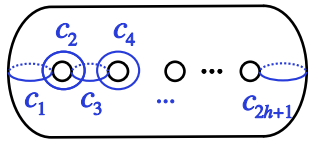}
\caption{The curves $c_1, \ldots, c_{2h+1}$}
\label{fig:Sigma_h_c}
\end{center}
\end{figure}

\subsection{The manifolds $Z_h$}
Let $Z_h$ denote the total space of the genus $h$ Lefschetz fibration with monodromy given by the full odd chain relation $(t_{c_1} t_{c_2} \cdots t_{c_{2h+1}})^{2h+2}=1$ in $\mod(\Sigma_h)$. These manifolds are simply connected, with Euler characteristic $e(Z_h)=2(2h^2+h+3)$ and signature $\sigma(Z_h)= -2(h+1)^2$ (\cite{EndoNagami}). The odd chain relation lifts to the relation $(t_{c_1} t_{c_2} \cdots t_{c_{2h+1}})^{2h+2}=t_\delta t_{\delta'}$ in $\mod(\Sigma_h^2)$, where $\delta$ and $\delta'$ are curves parallel to the boundary components. Consequently, the fibration on $Z_h$ admits two sections of square $-1$, and blowing these down results in a pencil on a manifold denoted $Z_h'$. The manifolds $Z_h$ and $Z_h'$ are complex surfaces of general type for $h \geq 2$, and $Z_h'$ is spin if and only if $h$ is even.

Let $C_n$ denote the ribbon surface in Figure~\ref{fig:Cbraid}.
\begin{figure}[ht]
\begin{center}
\includegraphics[scale=.9]{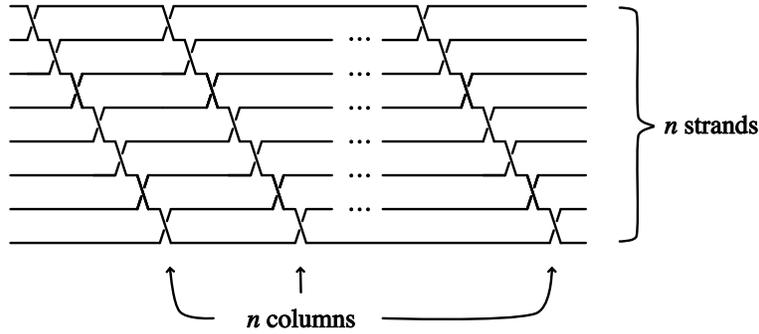}
\caption{The ribbon surface $C_n$}
\label{fig:Cbraid}
\end{center}
\end{figure}
Using the methods of \cite{Hyperelliptic}, we can describe $Z_h$ as the 2-fold branched cover of $\cp\#\cpb$ branched over the surface in Figure~\ref{fig:Z_h}.
\begin{figure}[ht]
\begin{center}
\includegraphics[scale=.9]{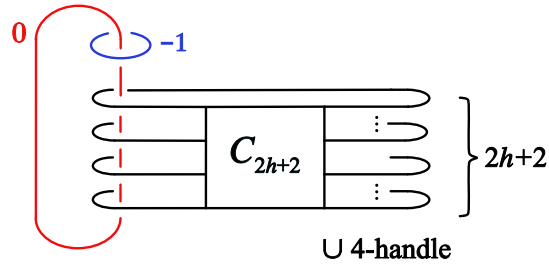}
\caption{The 2-fold branched cover is $Z_h$.}
\label{fig:Z_h}
\end{center}
\end{figure}
Briefly summarized, in this Figure the branch surface is a closed surface described as a banded unlink diagram, so that the full surface consists of the visible ribbon surface in Figure~\ref{fig:Z_h} union $2h+2$ invisible disks in the 4-handle attached to the boundary of the ribbon surface. The framing of the $-1$-framed 2-handle, as well as the fact that it does not intertwine with the ribbon branch surface, follows by considering the projection of the odd chain relation to the mapping class group $\mod(\Sigma_{0,2h+2})$ under the 2-fold branched cover $\Sigma_h \to \Sigma_{0,2h+2}$. Each Dehn twist $t_{c_i}$ projects to a standard generator $\sigma_i$ of $\mod(\Sigma_{0,2h+2})$ given by a right-handed disk twist about an arc that directly connects the $i$th and $(i+1)$st marked points. Let $D_{2h+2}^2$ be a disk in $\Sigma_{0,2h+2}$ that contains all of the marked branch points, and let $\ast$ be a point in the complement of  $D_{2h+2}^2$. The handle attachment can be understood by tracking a framed neighborhood of $\ast$ through an isotopy of the projection to the identity. (See \cite{F:Unchaining} for a more complete explanation of this detail.)
Here, the projection $(\sigma_1 \cdots \sigma_{2h+1})^{2h+2}$ is isotopic to a right-handed Dehn twist about a simple closed curve that encloses all of the marked branch points; this is isotopic to the identity on the sphere by an isotopy that fixes $\ast$ while rotating a neighborhood of $\ast$ once in a left-handed direction.  

\subsection{The manifolds $H_h$}
Let $H_h$ denote the genus $h$ Lefschetz fibration with monodromy given by the hyperelliptic relation $(t_{c_1} t_{c_2} \cdots t_{c_{2h+1}} t_{c_{2h+1}} \cdots t_{c_2} t_{c_1})^2=1$ in $\mod(\Sigma_h)$. It is well-known that $H_h$ is diffeomorphic to the rational surface $\cp \# (4h+5)\cpb$, from which it easily follows that $H_h$ is simply connected with Euler characteristic $e(H_h)=4(h+2)$ and signature $\sigma(H_h)= -4(h+1).$

We denote the (untwisted) fiber sum of $r$ copies of $H_h$ as $H_h(r)$, and use a similar notation for fiber sums of $Z_h$. We next describe an equivalent factorization for the Lefschetz fibrations in the family $H_h(r)$. 

\begin{Lem}\label{lem:hyper_equiv}
For all $n \geq 1$, in $\mod(\Sigma_h)$ we have $$(t_{c_1} t_{c_2} \cdots t_{c_{2h+1}} t_{c_{2h+1}} \cdots t_{c_2} t_{c_1})^n = (t_{c_1} t_{c_2} \cdots t_{c_{2h+1}})^n (t_{c_{2h+1}} \cdots t_{c_2} t_{c_1})^n.$$
\end{Lem}

\begin{proof}
Induction on $n$. The case where $n=1$ is a tautology. Assuming the statement for $n-1$, we then have
\begin{eqnarray*}
& & (t_{c_1} t_{c_2} \cdots t_{c_{2h+1}} t_{c_{2h+1}} \cdots t_{c_2} t_{c_1})^{n} \\
&=& (t_{c_1} t_{c_2} \cdots t_{c_{2h+1}} t_{c_{2h+1}} \cdots t_{c_2} t_{c_1})^{n-1} (\underline{t_{c_1} t_{c_2} \cdots t_{c_{2h+1}}} t_{c_{2h+1}} \cdots t_{c_2} t_{c_1}) \\
&=& (t_{c_1} t_{c_2} \cdots t_{c_{2h+1}})(t_{c_1} t_{c_2} \cdots t_{c_{2h+1}} t_{c_{2h+1}} \cdots t_{c_2} t_{c_1})^{n-1} (t_{c_{2h+1}} \cdots t_{c_2} t_{c_1}) \\
&=& (t_{c_1} t_{c_2} \cdots t_{c_{2h+1}}) (t_{c_1} t_{c_2} \cdots t_{c_{2h+1}})^{n-1} (t_{c_{2h+1}} \cdots t_{c_2} t_{c_1})^{n-1} (t_{c_{2h+1}} \cdots t_{c_2} t_{c_1}) \\
 &=& (t_{c_1} t_{c_2} \cdots t_{c_{2h+1}})^{n} (t_{c_{2h+1}} \cdots t_{c_2} t_{c_1})^{n}.\\
\end{eqnarray*}
The second equality follows from repeated use of the fact that $t_{c_i}$ commutes with  $$t_{c_1} t_{c_2} \cdots t_{c_{2h+1}} t_{c_{2h+1}} \cdots t_{c_2} t_{c_1}$$ for all $i$, which is used to move the underlined term to the left. The third equality is the inductive hypothesis.
\end{proof}

Lemma~\ref{lem:hyper_equiv} implies that $H_h(r)$ is equivalent to the Lefschetz fibration with factorization $$(t_{c_1} t_{c_2} \cdots t_{c_{2h+1}})^{2r} (t_{c_{2h+1}} \cdots t_{c_2} t_{c_1})^{2r}.$$ To depict $H_h(r)$ as a 2-fold branched cover of a Hirzebruch surface using this factorization, we must understand the projection of this factorization to an element of $\mod(\Sigma_{0,2h+2})$. As before, let $D_{2h+2}^2$ be a disk that contains all of the marked points, and let $\ast$ be another marked point in the complement of  $D_{2h+2}^2$. Using the Alexander method (see \cite{FarbMargalit}), the following is shown in \cite{alexander}.

\begin{Lem}\label{lem:isotopy} 
For all $1 \leq r \leq h$, $(\sigma_1 \cdots \sigma_{2h+1})^{2r} (\sigma_{2h+1} \cdots \sigma_{1})^{2r}$ is isotopic to the identity on $\Sigma_{0,2h+2}$ by an isotopy that moves $\ast$ along a path that passes once between the $(2r)$th and $(2r+1)$st marked points without twisting its neighborhood.
\end{Lem} 

As a result, we see that $H_h(r)$ is the 2-fold branched cover of $S^2 \times S^2$ branched over the ribbon surface shown in Figure~\ref{fig:H_h_r}. The boxed ribbon $E_{2r}$ is defined in the previous Figure~\ref{fig:Ebraid}.
\begin{figure}[ht]
\begin{center}
\includegraphics[scale=.9]{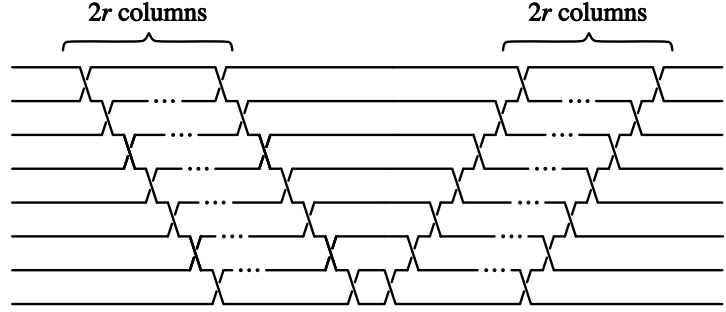}
\caption{The ribbon surface $E_{2r}$}
\label{fig:Ebraid}
\end{center}
\end{figure}
\begin{figure}[ht!]
\begin{center}
\includegraphics[scale=.9]{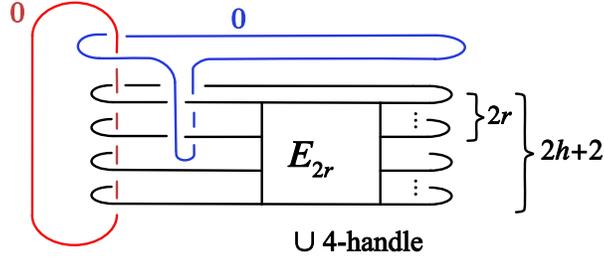}
\caption{The 2-fold branched cover is $H_h(r)$.}
\label{fig:H_h_r}
\end{center}
\end{figure}

\subsection{More Fiber Sums}

We can also consider fiber sums between the fibrations on $Z_h$ and $H_h$. The following proposition is due to Hisaaki Endo.

\begin{Prop}[\cite{Endo_chakiris}, Lemma 4.1] \label{prop:hyperelliptic_equiv}
The fiber sums $Z_h(2)$ and $H_h(h+1)$ are isomorphic as Lefschetz fibrations.
\end{Prop}

As a result, fiber sums between any number of copies of $Z_h$ and $H_h$ are isomorphic to either $H_h(m)$ or $Z_h(1) \#_f H_h(m-1)$, for some $m$. In both cases, the result is spin if and only if $h$ is odd and $m$ is even \cite{EndoNagami}.

\section{The manifolds $X_{g,h}'[i]$} \label{sec:manifolds}

We next derive as Theorem~\ref{thm:relation} below a relation in the mapping class group $\mod(\Sigma_g^{2})$ which will be used to construct symplectic Lefschetz pencils. This relation, and the corresponding pencils, are a direct generalization of the relations derived in \cite{BHM:Unchaining}. Instead of adapting the algebraic derivation there, we will proceed by finding a relation in the braid group $\mod(\Sigma_{0,2g+2}^1)$, which is lifted to a relation in $\mod(\Sigma_g^{2})$ using a 2-fold branched cover.  We begin with a Lemma due to Kenneth Chakiris (\cite{Chakiris}, Lemma 3.5.)

\begin{Lem}[The Reversing Lemma]\label{lem:reversing}
Let $\sigma_1, \ldots, \sigma_{2g+1}$ be the standard generators of $\mod(\Sigma_{0,2g+2}^1)$. For any $2 \leq m \leq 2g+1$, we have
$(\sigma_1 \cdots \sigma_{m})^{m+1} = (\sigma_m \cdots \sigma_{1})^{m+1}.$
\end{Lem}

Let $1 \leq h < g$. For all $1 \leq j \leq 2g-2h$, let $d_j$ and $e_j$ denote the curves on $\Sigma_g^2$ shown in Figure~\ref{fig:Sigma_g}. We abbreviate
 $D_{2h+2}=t_{d_{1}} t_{d_{2}} \cdots t_{d_{2g-2h}}$ and $E_{2h+2}=t_{e_{2g-2h}} \cdots t_{e_{2}} t_{e_{1}}.$ 
\begin{figure}[ht!]
\begin{center}
\includegraphics[scale=1]{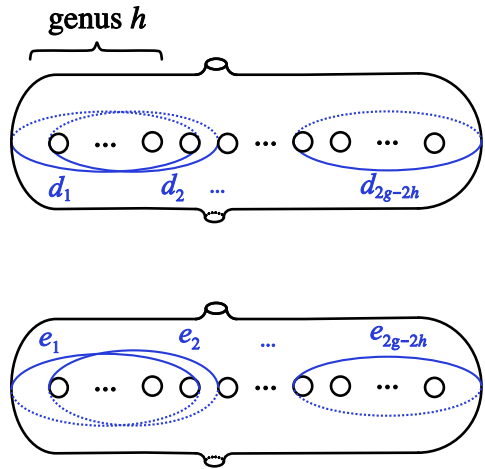}
\caption{The surface $\Sigma_g^2$}
\label{fig:Sigma_g}
\end{center}
\end{figure}

\begin{Thm}\label{thm:relation} 
Let $1 \leq h < g$, and write $2g+2$ as $2g+2=p(2h+2)+2r$, with $p>0$ and $0 \leq r <h+1$. Then in $\mod(\Sigma_{g}^2)$ we have
\begin{equation*} t_{\delta} t_{\delta'} = D_{2h+2} E_{2h+2}  (t_{c_1} \cdots t_{c_{2h+1}})^{(2h+2)(2p-1)} (t_{c_{2h+3}}  \cdots t_{c_{2g+1}})^{2g-2h} (t_{c_1} \cdots t_{c_{2h+1}})^{2r} (t_{c_{2h+1}}  \cdots t_{c_{1}})^{2r}, \end{equation*}
where $\delta$ and $\delta'$ are curves parallel to the boundary components of $\Sigma_{g}^2$.
\end{Thm}

\begin{Rem} When $r=0$, the terms $(t_{c_1} \cdots t_{c_{2h+1}})^{2r} (t_{c_{2h+1}}  \cdots t_{c_{1}})^{2r}$ are omitted from the factorization.
\end{Rem}

 \begin{proof}
We give a proof using braids, in the spirit of Lemma 8 of \cite{Auroux_stable}, by identifying the braid group $B_{2g+2}$ with the mapping class group $\mod(\Sigma^1_{0,2g+2})$ of the disk with marked points. Define $\tau_{i}$ and $\upsilon_{i}$ to be the braids shown in Figure~\ref{fig:qp_gens}.
\begin{figure}[ht!]
\begin{center}
\includegraphics[scale=.9]{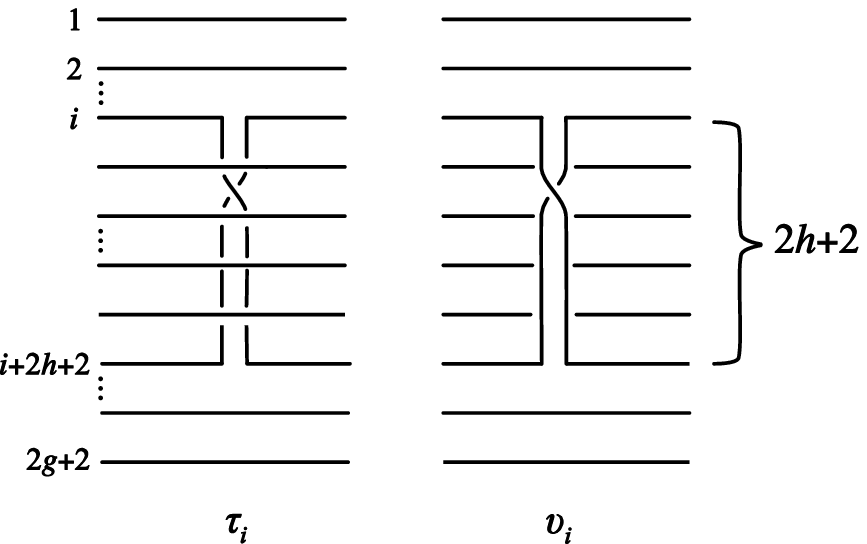}
\caption{The braids $\tau_{i}$ and $\upsilon_{i}$}
\label{fig:qp_gens}
\end{center}
\end{figure}

 \begin{figure}[ht!]
\begin{center}
\includegraphics[scale=.9]{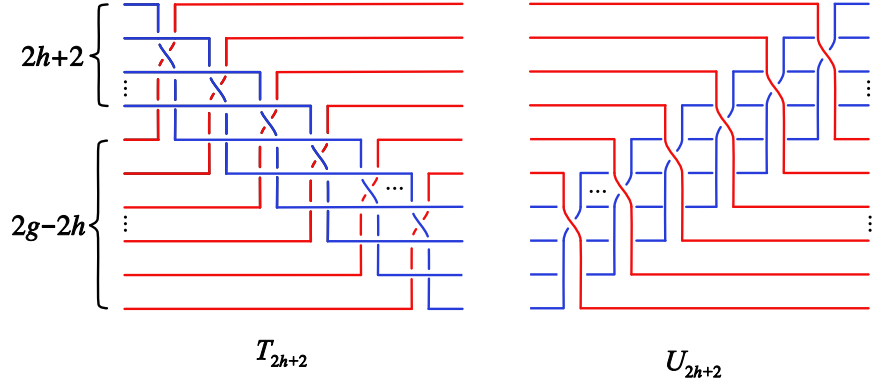}
\caption{The braids $T_{2h+2}$ and $U_{2h+2}$ }
\label{fig:TUbraid}
\end{center}
\end{figure}

Consider $T_{2h+2}=\tau_{1} \tau_{2} \cdots \tau_{2g-2h}$ and $U_{2h+2}=\upsilon_{2g-2h} \cdots \upsilon_{2} \upsilon_{1},$ shown in Figure~\ref{fig:TUbraid}. It is easy to see that $T_{2h+2}$ passes the upper $2h+2$ strands over the lower $2g-2h$ strands, forming a right half twist between the two collections. (This is highlighted by the colors in Figure~\ref{fig:TUbraid}.) The upper strands, regarded by themselves, form the braid $(\osigma_1 \cdots \osigma_{2h+1})^{2g-2h}$, while the lower strands are parallel. Similarly, for $U_{2h+2}$, the upper $2g-2h$ strands pass over the lower $2h+2$ strands, forming a right half twist between the two collections. The lower stands, regarded by themselves, form the braid $(\osigma_{2h+1} \cdots \osigma_{1})^{2g-2h}$, with the upper strands parallel. 

Adding additional generators to the product $T_{2h+2}U_{2h+2}$, we claim that 
$$T_{2h+2} U_{2h+2} (\sigma_1 \cdots \sigma_{2h+1})^{2g-2h} (\sigma_{2h+1} \cdots \sigma_{1})^{2g+2}  (\sigma_{2h+3} \cdots \sigma_{2g+1})^{2g-2h}$$
forms a full right twist among all strands.

This can be seen by expanding this expression as:
\begin{equation*}
T_{2h+2} U_{2h+2} (\sigma_1 \cdots \sigma_{2h+1})^{2g-2h} (\sigma_{2h+1} \cdots \sigma_{1})^{2g-2h} (\sigma_{2h+1} \cdots \sigma_{1})^{2h+2} (\sigma_{2h+3} \cdots \sigma_{2g+1})^{2g-2h}.
\end{equation*}
The additional terms $$(\sigma_1 \cdots \sigma_{2h+1})^{2g-2h} (\sigma_{2h+1} \cdots \sigma_{1})^{2g-2h}$$ following $T_{2h+2}U_{2h+2}$ undo the braiding among the upper $2h+2$ strands, and the subsequent terms $$(\sigma_{2h+1} \cdots \sigma_{1})^{2h+2} (\sigma_{2h+3} \cdots \sigma_{2g+1})^{2g-2h}$$ separately introduce a full right twist within both the upper $2h+2$ and lower $2g-2h$ strands. This results in a full right twist among {\em all} strands, as claimed.

The assumed condition $2g+2=p(2h+2)+2r$ implies $2g-2h=(p-1)(2h+2)+2r$. Substituting both and using the Reversing Lemma we get
\begin{eqnarray*}
&\mbox{}& T_{2h+2} U_{2h+2} (\sigma_1 \cdots \sigma_{2h+1})^{2g-2h} (\sigma_{2h+1} \cdots \sigma_{1})^{2g+2}  (\sigma_{2h+3} \cdots \sigma_{2g+1})^{2g-2h} \\
&=& T_{2h+2} U_{2h+2} (\sigma_1 \cdots \sigma_{2h+1})^{(2h+2)(p-1)+2r} (\sigma_{2h+1} \cdots \sigma_{1})^{(2h+2)p+2r}  (\sigma_{2h+3} \cdots \sigma_{2g+1})^{2g-2h} \\
&=& T_{2h+2} U_{2h+2} (\sigma_1 \cdots \sigma_{2h+1})^{(2h+2)(2p-1)} (\sigma_1 \cdots \sigma_{2h+1})^{2r}(\sigma_{2h+1} \cdots \sigma_{1})^{2r} (\sigma_{2h+3} \cdots \sigma_{2g+1})^{2g-2h}. \\
\end{eqnarray*}
(Note that if $r=0$, the terms with exponent $2r$ become unnecessary.)

Since this braid is a full right twist, interpreting the braid as an element of the mapping class group $\mod(\Sigma^1_{0,2g+2})$ of the disk with $2g+2$ marked points, this expression is in the same class as a Dehn twist about a curve parallel to the boundary (\cite{FarbMargalit}). The result follows by lifting this relation under the 2-fold branched cover $\Sigma^2_{g} \to \Sigma^1_{0,2g+2}$ with $2g+2$ branch points; the lifts of $T_{2h+2}$ and $U_{2h+2}$ are $D_{2h+2}$ and $E_{2h+2}$, respectively, and each $\sigma_i$ lifts to $t_{c_i}$.
\end{proof}

Starting from the factorization in Theorem~\ref{thm:relation}, for any $0 \leq i \leq 2p-1$ we can perform $i$ unchaining surgeries to the subword $ (t_{c_1} t_{c_2} \cdots t_{c_{2h+1}})^{2h+2}$, and another one to the subword $(t_{c_{2h+3}}  \cdots t_{c_{2g+1}})^{2g-2h}$. This yields the relation 
$$ D_{2h+2} E_{2h+2} t_a^i  t_{a'}^i t_b t_{b'} (t_{c_1} \cdots t_{c_{2h+1}})^{(2h+2)(2p-1-i)} (t_{c_1} \cdots t_{c_{2h+1}})^{2r} (t_{c_{2h+1}}  \cdots t_{c_{1}})^{2r} =t_{\delta} t_{\delta'} $$ 
in $\mod(\Sigma_{g}^2)$, where $a, a', b, b'$ are shown in Figure~\ref{fig:Sigma_g_ab}.

\begin{figure}[ht]
\begin{center}
\includegraphics[scale=1]{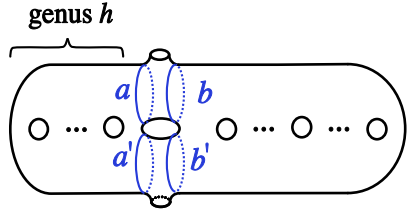}
\caption{The curves $a, a', b, b'$ on $\Sigma_g^2$}
\label{fig:Sigma_g_ab}
\end{center}
\end{figure}

The clever inductive argument using the lantern relation found in the proof of Theorem 4.6 of \cite{BHM:Unchaining} can now be applied verbatim to lift this relation to $\mod(\Sigma_g^{2(i+1)})$ as 
\begin{multline}
\Delta = D_{2h+2} E_{2h+2}  (t_{x_{i+1}} \cdots t_{x_{2}} t_{x_{1}}) (t_{x_{i+1}^\prime} \cdots t_{x_{2}^\prime}t_{x_{1}^\prime}) (t_{c_1} t_{c_2} \cdots t_{c_{2h+1}})^{(2h+2)(2p-i-1)} \\ (t_{c_1} t_{c_2} \cdots t_{c_{2h+1}})^{2r}(t_{c_{2h+1}} \cdots t_{c_2} t_{c_1})^{2r},
\label{eqn:monodromy}
\end{multline}
where $\Delta$ is the product of right-handed Dehn twists about a curve parallel to each boundary component of $\Sigma_g^{2(i+1)}$. The curves $x_k$ and $x'_k$ are shown in Figure~\ref{fig:Sigma_g_x}.

\begin{figure}[ht]
\begin{center}
\includegraphics[scale=1]{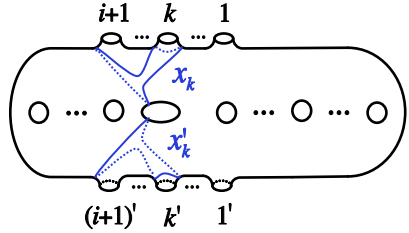}
\caption{The curves $x_k, x_k'$ on $\Sigma_g^{2(i+1)}$}
\label{fig:Sigma_g_x}
\end{center}
\end{figure}

This factorization gives the following theorem.

\begin{Thm} \label{thm:manifolds}
Let $1 \leq h <g$, and write $2g+2$ as $2g+2=p(2h+2)+2r$, with $p>0$ and $0 \leq r <g+1$. Then
for all $0 \leq i \leq 2p-1$, there is a symplectic 4-manifold $X'_{g,h}[i]$ which admits a genus $g$ Lefschetz pencil with monodromy factorization in $\mod(\Sigma_g^{2(i+1)})$ given by Equation \ref{eqn:monodromy}.
\end{Thm}

\begin{Rem}
Setting $h=1$ in Equation~\ref{eqn:monodromy} reproduces Baykur-Hayano-Monden's relations in Theorem 4.6 of \cite{BHM:Unchaining} and consequently their manifolds correspond to our manifolds $X_{g,1}'[i]$. In their description, separate factorizations are given for $g$ even and odd. Thus Theorem~\ref{thm:relation} simultaneously consolidates their factorization into a single statement for both cases while also generalizing it to hold for all $h$.
\end{Rem}

\begin{Prop} \label{prop:properties}
The manifolds $X'_{g,h}[i]$ have Euler characteristic $$e(X'_{g,h}[i])=4-4h+2(2h+1)(2g+2)-(i+1)(2h+1)(2h+2)$$
and signature $$\sigma(X'_{g,h}[i])=-(2h+2)(2g+2)+2(i+1)(h+1)^2.$$
\end{Prop}

\begin{proof}
These formulas can be seen in several ways, and we leave the details to the reader. For instance, the Euler characteristic follows from translating the factorization in Equation~\ref{eqn:monodromy} into the standard handlebody description of the Lefschetz fibration $X_{g,h}[i]$ obtained by blowing up $X'_{g,h}[i]$. The signature can be seen by using Endo's fractional signature to first compute the signature of the Lefschetz fibration defined by the factorization in Theorem~\ref{thm:relation} (see \cite{Endo_fractional}), and then using Endo and Nagami's results from \cite{EndoNagami} to obtain the signature of manifolds obtained by unchaining surgery.  Both formulas can also be corroborated using the branched covering description of $X'_{g,h}[i]$ developed in the next section.
\end{proof}

\begin{Rem}
At this point, it is straightforward to use the monodromy factorization in Equation~\ref{eqn:monodromy} to determine the fundamental groups of $X'_{g,h}[i]$ (they are mostly simply connected), as well as to state conditions on $g,h,$ and $i$ that determine when $X'_{g,h}[i]$ is spin, using the technique developed in \cite{BHM:Unchaining}. However, these properties also follow easily from our Theorem~\ref{thm:main} below, and we defer these statements to Corollary~\ref{cor:spin}.
\end{Rem}

\subsection{A branched cover description of $X'_{g,h}[i]$}

We now describe the manifolds  $X'_{g,h}[i]$ as a 2-fold branched covers of a rational surface. We begin by considering the Lefschetz fibration $X_{g,h}[i]$ obtained by blowing up the pencil on $X'_{g,h}[i]$ at its base points. Under the capping homomorphism $\mod(\Sigma_g^{2(i+1)}) \to \mod({\Sigma_g})$, the various curves $x_k$ and $x_k'$ in Figure~\ref{fig:Sigma_g_x} become a common curve $x$ and $x'$ and we have
$$D_{2h+2} E_{2h+2}  t_{x}^{i+1} t_{x^\prime}^{i+1} \ (t_{c_1} t_{c_2} \cdots t_{c_{2h+1}})^{(2h+2)(2p-i-1)} (t_{c_1} t_{c_2} \cdots t_{c_{2h+1}})^{2r}(t_{c_{2h+1}} \cdots t_{c_2} t_{c_1})^{2r}$$
is the identity in $\mod(\Sigma_g)$. From this relation, we express $X_{g,h}[i]$ as the 2-fold cover of $\F_1 \# (i+1) \cpb$ branched over the surface shown in Figure~\ref{fig:base_0}. In this diagram, $A_{2h+2}$ represents the ribbon surface bounded by the braid $T_{2h+2}U_{2h+2}$ as shown in Figure~\ref{fig:TUbraid}. The exponent of $(C_{2h+2})^{2p-i-1}$ denotes multiple copies of the ribbon surface $C_{2h+2}$. The red $-1$-framed 2-handles each lift to a pair of 2-handles corresponding to a pair of Dehn twists $t_x$ and $t_{x'}$ in the monodromy factorization for $X_{g,h}[i]$. The other terms in the factorization correspond to 2-handles that are lifts of bands in the various ribbon surfaces pictured. The blue $-1$-framed 2-handle appears as shown because the projection of the monodromy to $\Sigma_{0,2g+2}$ equals a single right-handed Dehn twist about all branch points; this fact follows from its derivation in Theorem~\ref{thm:relation}, and is also confirmed with the Alexander method in \cite{hadi_thesis}. The branch surface includes $2g+2$ disks in the 4-handle attached to the boundary of the ribbon surface pictured.
\begin{figure}[ht]
\begin{center}
\includegraphics[width=6in]{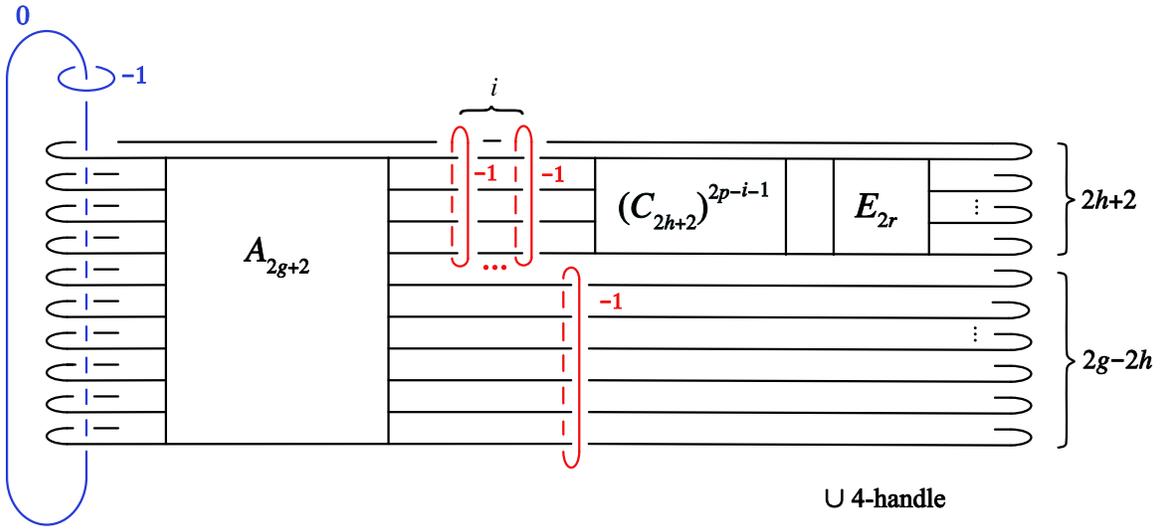}
\caption{The 2-fold branched cover is $X_{h,g}(i)$.}
\label{fig:base_0}
\end{center}
\end{figure}

We next modify the base of this branched covering, for use in the next section. We begin by repositioning the red $-1$-framed 2-handles by moving them to the right and swinging them around the back of the ribbon surface so that they appear on the left, as shown in Figure~\ref{fig:base_1}. 
\begin{figure}[ht]
\begin{center}
\includegraphics[width=6in]{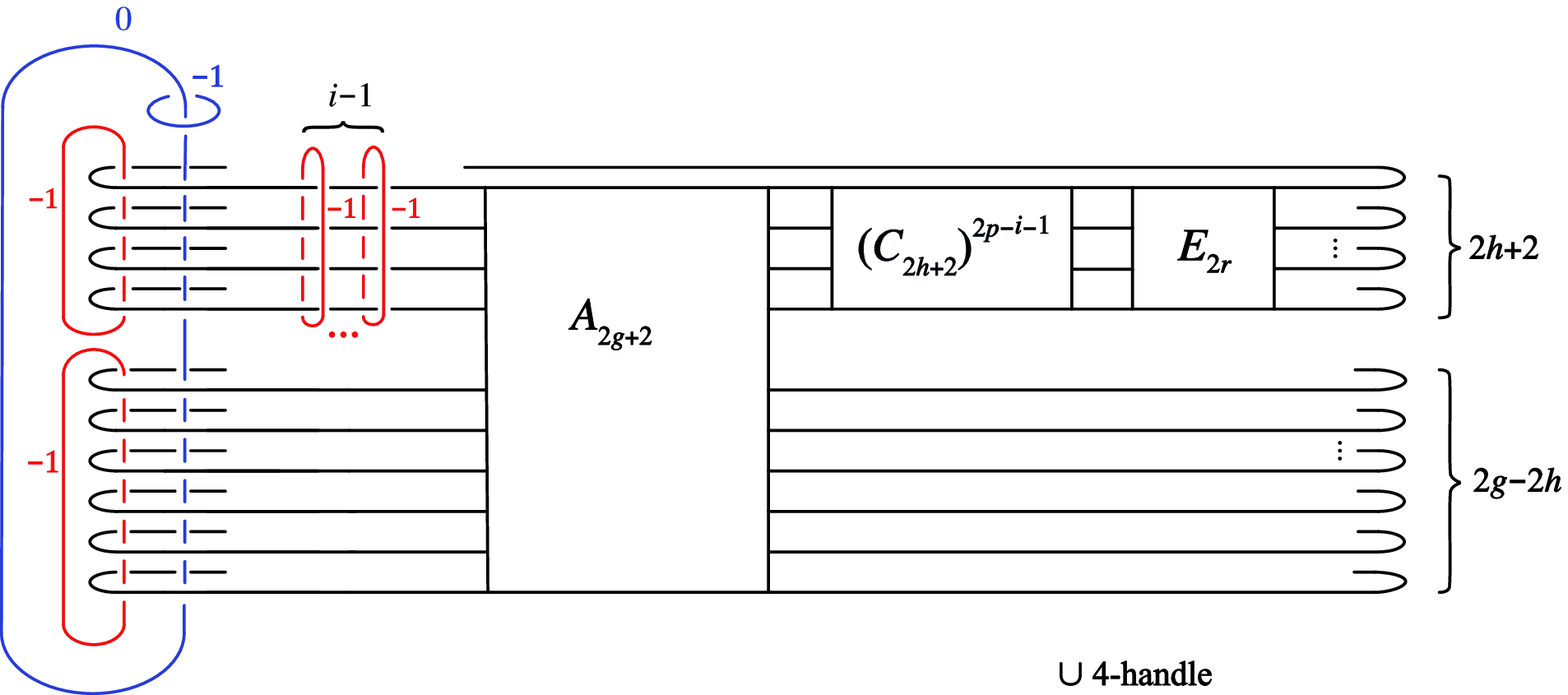}
\caption{}
\label{fig:base_1}
\end{center}
\end{figure}

We now slide one of the upper $-1$-framed 2-handle over the lower one, giving Figure~\ref{fig:base_2}. 
Figures \ref{fig:base_2} through \ref{fig:base_4} then show a sequence of handle slides, indicated by the grey bands in each picture. The steps from Figure~\ref{fig:base_1} to \ref{fig:base_4} can then be repeated for each of the $i-1$ many $-1$-framed 2-handles at the top of Figure~\ref{fig:base_1}, resulting in Figure~\ref{fig:base_5}. Two more handle slides bring us to Figure~\ref{fig:base_6}.
\begin{figure}[ht]
\centering
\begin{minipage}{.45\textwidth}
  \centering
  \includegraphics[scale=.75]{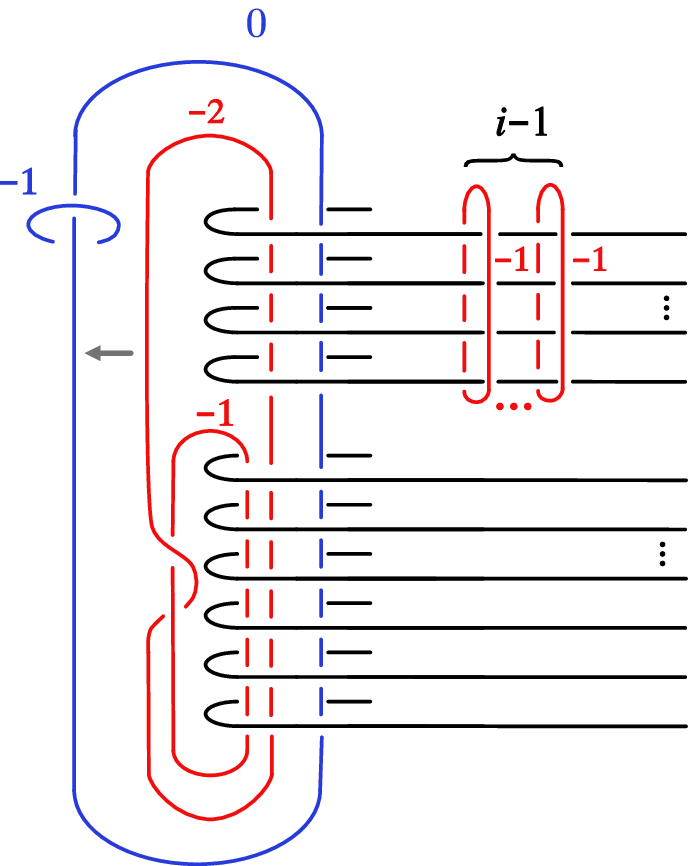}
  \caption{}
  \label{fig:base_2}
\end{minipage} \hfill
\begin{minipage}{.45\textwidth}
  \centering
  \includegraphics[scale=.75]{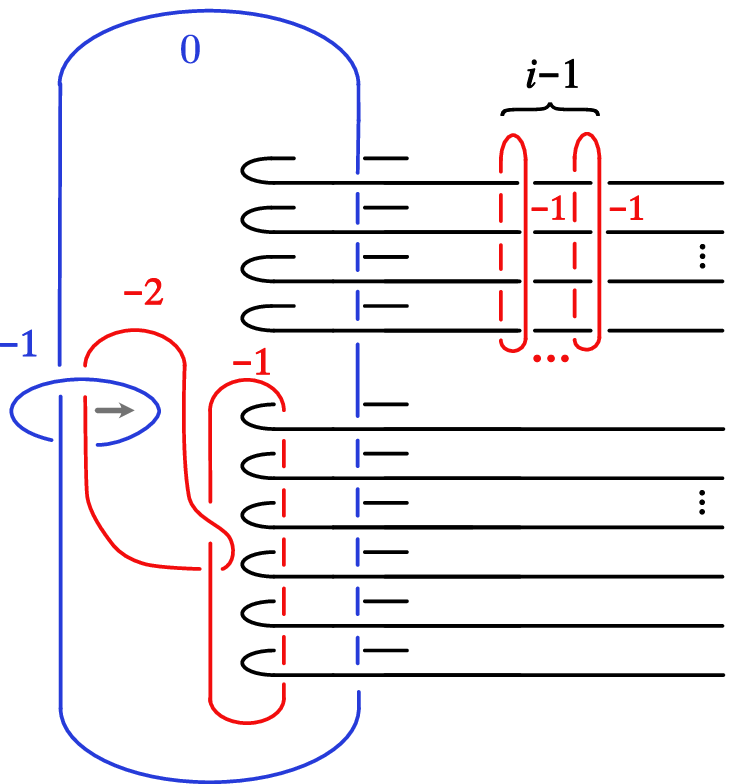}
  \caption{}
  \label{fig:base_3}
\end{minipage}
\end{figure}

\begin{figure}[ht]
\centering
\begin{minipage}{.45\textwidth}
  \centering
  \includegraphics[scale=.75]{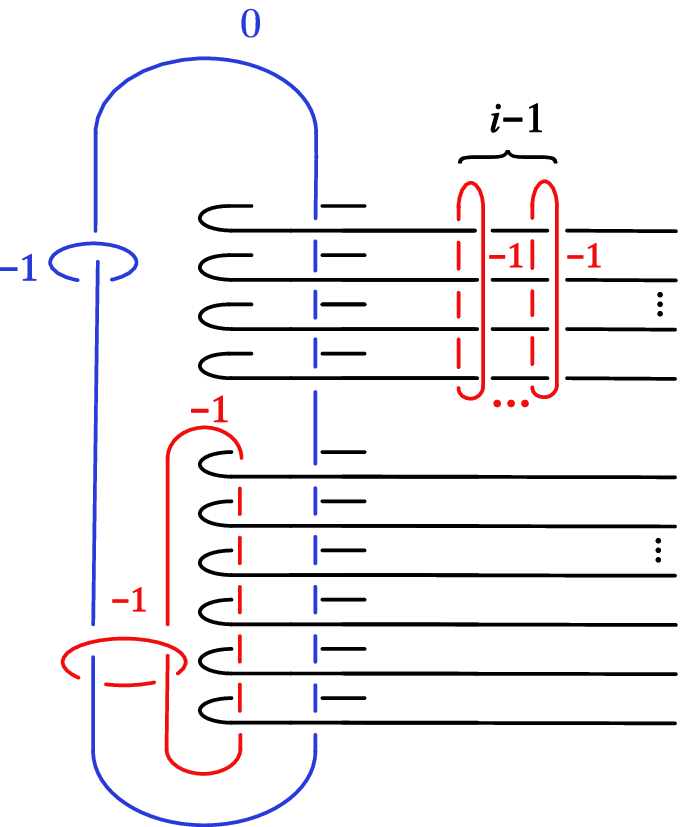}
  \caption{}
  \label{fig:base_4}
\end{minipage} \hfill
\begin{minipage}{.45\textwidth}
  \centering
  \includegraphics[scale=.75]{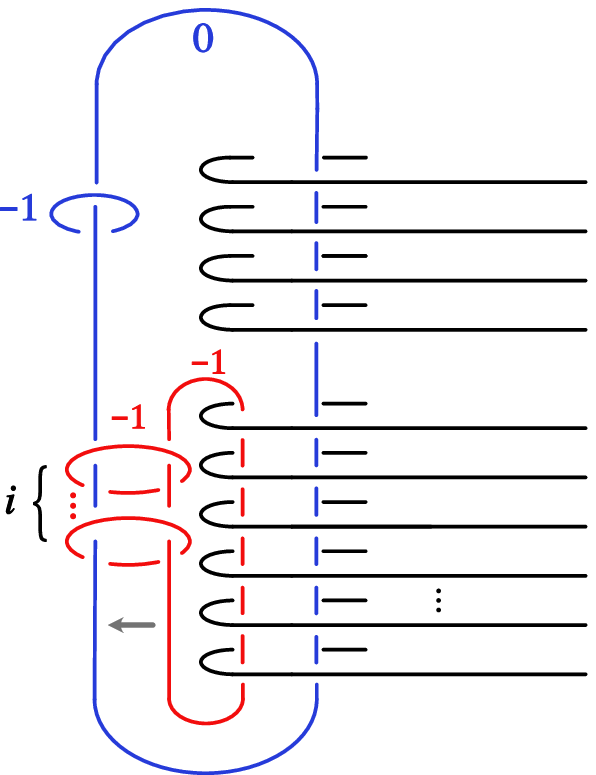}
  \caption{}
  \label{fig:base_5}
\end{minipage}
\end{figure}

\begin{figure}[ht]
\centering
\begin{minipage}{.45\textwidth}
  \centering
  \includegraphics[scale=.75]{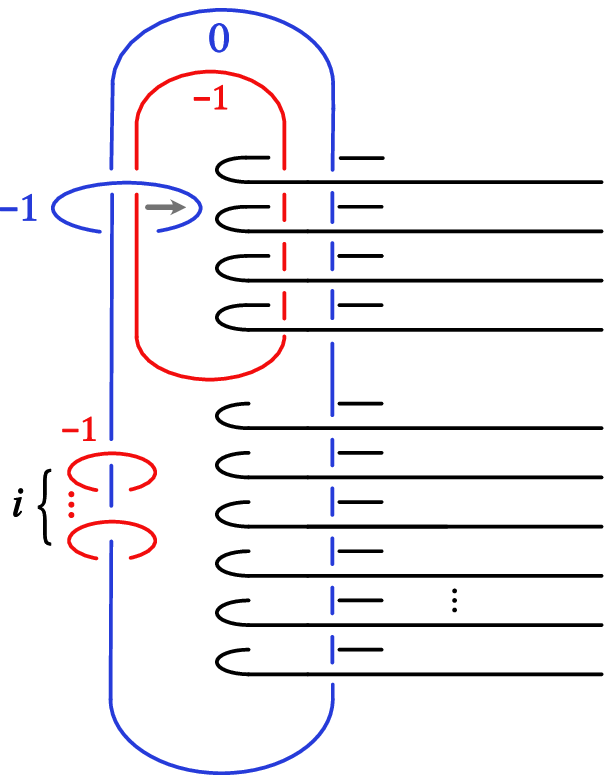}
  \caption{}
  \label{fig:base_5_b}
\end{minipage} \hfill
\begin{minipage}{.45\textwidth}
  \centering
  \includegraphics[scale=.75]{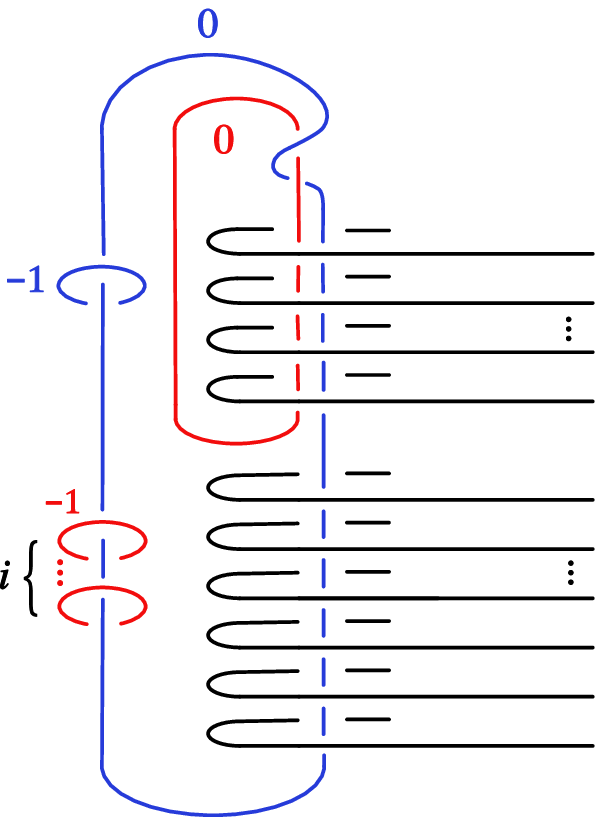}
  \caption{}
  \label{fig:base_6}
\end{minipage}
\end{figure}

We blow down each of the exceptional spheres in Figure~\ref{fig:base_6} given by the $-1$-framed 2-handles to arrive at Figure~\ref{fig:base_7}. We claim that the 2-fold branched cover of $\F_{i+1}$ branched over the pictured surface is $X_{g,h}'[i]$. Since this cover was obtained by blowing down $X_{g,h}[i]$ a total of $2(i+1)$ times, to justify this statement, we must show that each of the spheres blown down are sections of the fibration on $X_{g,h}[i]$. This can be seen by observing that the 0-framed 2-handle is fixed by the diffeomorphism that is the transition from Figure~\ref{fig:base_1} to Figure~\ref{fig:base_6}. These spheres are $-1$-framed meridians to this 2-handle in Figure~\ref{fig:base_6}, and thus when those spheres are traced back to Figure~\ref{fig:base_1} they lift to two sections of the fibration on $X_{g,h}[i]$ of square $-1$.

\begin{figure}[ht]
\begin{center}
\includegraphics[width=6in]{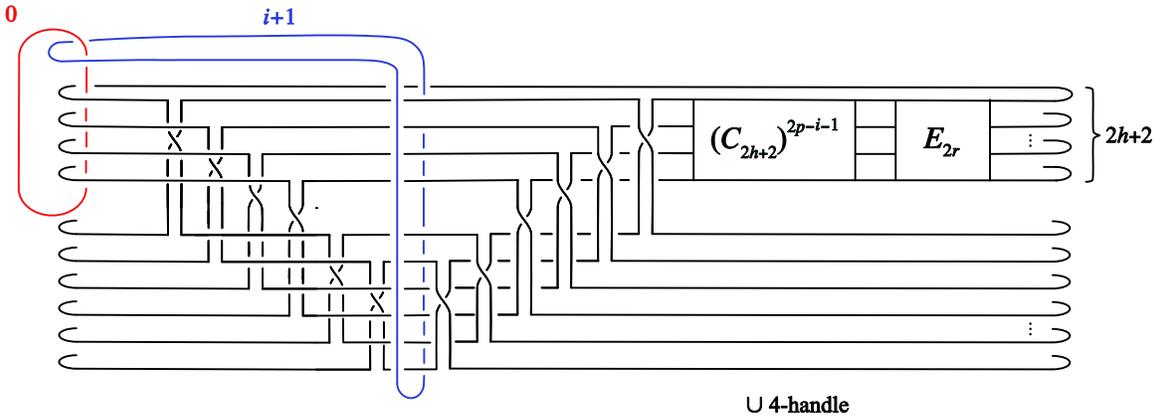}
\caption{The 2-fold branched cover is $X_{g,h}'[i]$.}
\label{fig:base_7}
\end{center}
\end{figure}

\section{From pencils to fibrations} \label{sec:theorem}

We are now ready for our main theorem, which identifies the manifolds $X_{g,h}'[i]$ as complex manifolds obtained as fiber sums of the genus $h$ Lefschetz fibrations from Section~\ref{sec:hyperelliptic}.

\begin{Thm} \label{thm:main}
Let $1 \leq h < g$, and write $2g+2$ as $2g+2=p(2h+2)+2r$, with $p>0$ and $0 \leq r <h+1$.
\begin{itemize}
\item[(a)] For all $0 \leq i < 2p-1$, $X'_{g,h}[i]$ is diffeomorphic to  $Z_h(2p-1-i) \ \#_f \ H_h(r)$. If $r=0$, then $H_h(r)$ is omitted from this expression.
\item[(b)] If $i=2p-1$ and $r \neq 0$, then $X'_{g,h}[2p-1]$ is diffeomorphic to $H_h(r)$.
\item[(c)] If $i=2p-1$ and $r = 0$, then $X'_{g,h}[2p-1]$ is diffeomorphic to $\Sigma_h \times S^2$.
\end{itemize}
\end{Thm}

Before addressing the proof of Theorem~\ref{thm:main}, we record two corollaries.

\begin{Cor} \label{cor:spin}
Assume $0 \leq i < 2p-1$; or, $i=2p-1$ and $r \neq 0$.
\begin{itemize}
\item[(a)] $X'_{g,h}[i]$ is simply connected.
\item[(b)] $X'_{g,h}[i]$ is spin if and only if $h$ is odd and $g$ and $i$ have the same parity.
\end{itemize}
\end{Cor}

\begin{proof} The assumptions on $i$ and $r$ avoid case (c) of Theorem~\ref{thm:main}, so $X'_{g,h}[i]$ is diffeomorphic to a fiber sum of $Z_h$ and $H_h$. Statement (a) of the Corollary follows immediately from the fact that $Z_h$ and $H_h$ are simply connected.

For part (b), recall that any fiber sum involving $Z_h$ and $H_h$ is equivalent to one of $Z_h(1) \#_f H_h(m-1)$ or $H_h(m)$, with the result spin if and only if $h$ is odd and $m$ is even (\cite{EndoNagami}).

Assume for the rest of the proof that $h$ is odd. Dividing the equation $2g+2=p(2h+2)+2r$ by $2$ shows that under this assumption, $g$ and $r$ always have opposite parity. 

Assume $i$ is odd. Then Proposition~\ref{prop:hyperelliptic_equiv} implies that $Z_h(2p-1-i) \ \#_f \ H_h(r)$ is equivalent to $$H_h((p-\tfrac{i+1}{2})(h+1)+r).$$ This is spin if and only if the expression $(p-\frac{i+1}{2})(h+1)+r$ is even.  This happens if and only if $r$ is even, or equivalently if and only if $g$ is odd.

Assume $i$ is even. Then Proposition~\ref{prop:hyperelliptic_equiv} implies that $Z_h(2p-1-i) \ \#_f \ H_h(r)$ is equivalent to $$Z_h(1) \ \#_f \ H_h((p-\tfrac{i+2}{2})(h+1)+r).$$ This is spin if and only if the expression $(p-\frac{i+2}{2})(h+1)+r$ is odd. This happens if and only if $r$ is odd, or equivalently if and only if $g$ is even. 
\end{proof}

By regarding $g$ and $i$ in the expressions in Theorem~\ref{thm:main} as functions of $p$, we see that any given fiber sum with summands $Z_h$ and $H_h$ admits infinitely many Lefschetz pencils, whose genera grow without bound as an arithmetic sequence. To state this in general, observe that by Proposition~\ref{prop:hyperelliptic_equiv}, for $h \geq 2$ any such fiber sum can be uniquely expressed as $Z_h(q) \ \#_f \ H_h(r)$, for some $0 \leq r < h+1$ and $q \geq 0$. Equating the number of $Z_h$ summands with the expression in Theorem~\ref{thm:main} gives $q=2p-i-1$, or $i=2p-1-q$. The expression $2g+2=p(2h+2)+2r$ simplifies to $g=p(h+1)+r-1$. Thus we have 

\begin{Cor} \label{cor:infinite_pencils}
Let $h \geq 1$. Assume $0 \leq r < h+1$ and $q \geq 0$.
Set $g(p)=p(h+1)+r-1$ and $i(p)=2p-1-q$. Then $Z_h(q) \ \#_f \ H_h(r)$ is diffeomorphic to $X'_{g(p),h}[i(p)]$ for all $p \geq \max \{\frac12 (q+1),2 \}$.
\end{Cor}

\begin{Rem} The condition $p \geq \max \{\frac12 (q+1),2 \}$ guarantees that $i \geq 0$ and $g>h$ for all $p$.
\end{Rem}

\subsection{An isotopy lemma} \label{sec:isotopy_lemma}
To prove Theorem~\ref{thm:main}, we start with a lemma which describes an isotopy between two ribbon surfaces embedded in a Hirzebruch surface. The surfaces are given by banded unlink diagrams, and the isotopy makes use of some standard moves on these diagrams, a complete list of which is given in \cite{BandedUnlink}. We will make use of two iterations of these moves, previously described in  \cite{F:Unchaining}, where they were referred to as {\em band dives} and {\em 2-handle band dives}, respectively. These moves are summarized in Figures \ref{fig:band_dive} and \ref{fig:2h_band_dive}; they will be used by replacing the initial ribbon surface in each figure with the final one. 
\begin{figure}[ht!]
\begin{center}
\includegraphics[scale=.7]{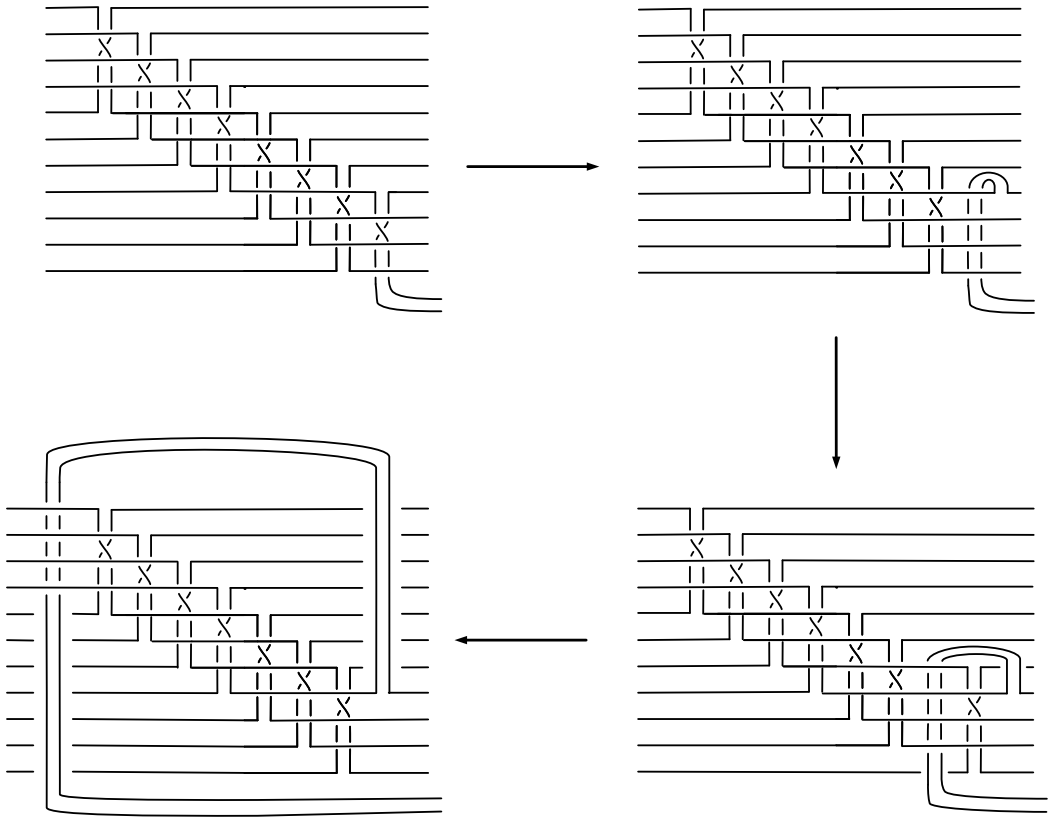}
\caption{A band dive}
\label{fig:band_dive}
\end{center}
\end{figure}
\begin{figure}[ht!]
\begin{center}
\includegraphics[width=5.5in]{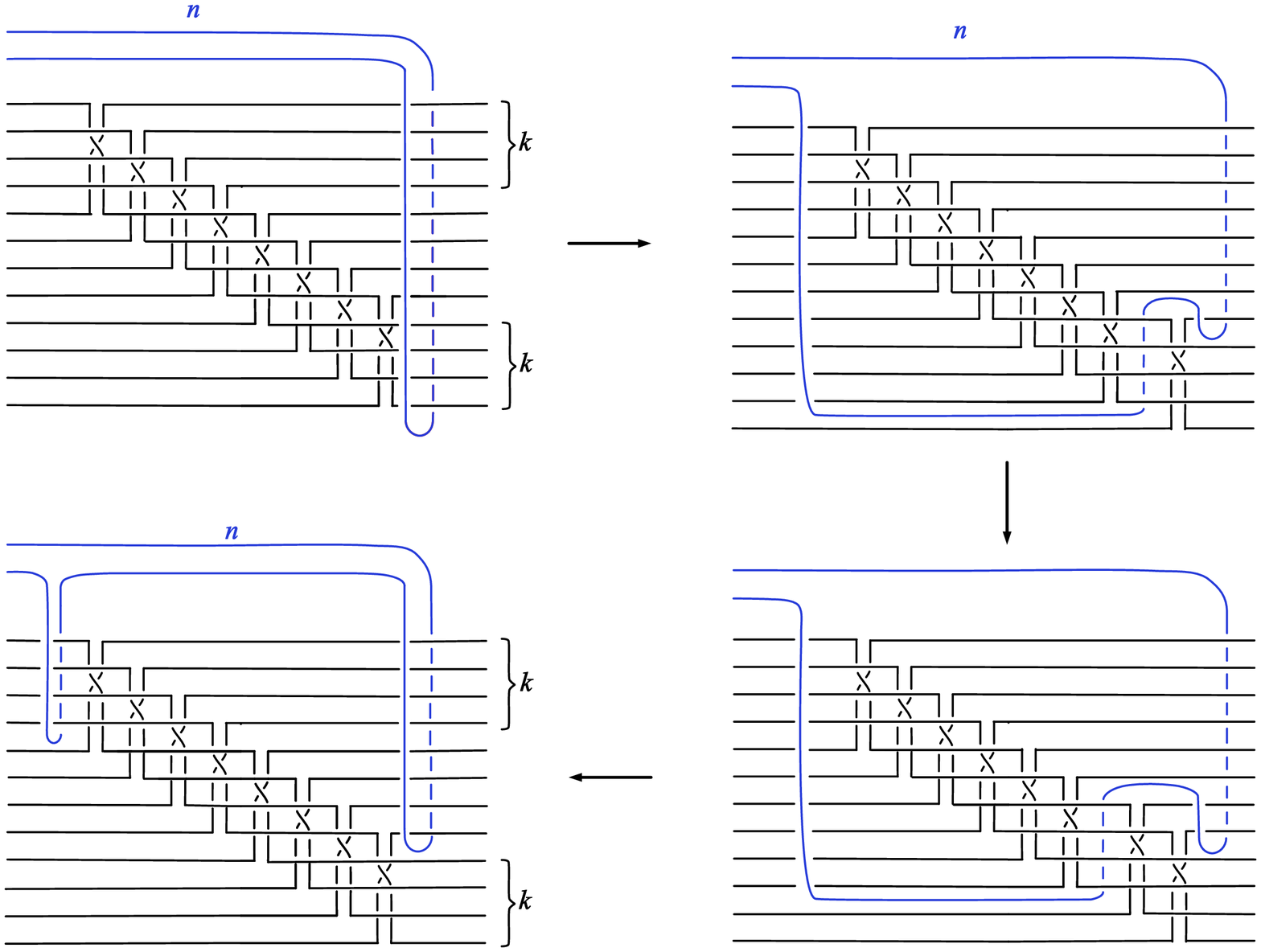}
\caption{A 2-handle band dive}
\label{fig:2h_band_dive}
\end{center}
\end{figure}

Let $\Sigma(R,S,T)$ denote a ribbon surface in the $n$th Hirzebruch surface $\F_n$ of the form shown in Figure~\ref{fig:lemma_0}. The ``length'' of each long band in the surface is a fixed constant $k$, which we suppress from the notation. The box represents any collection of bands subject to the requirement that they are confined to that region. The parameters $R, S$ and $T$ are explained in the caption to Figure~\ref{fig:lemma_0}.

\begin{figure}[ht!]
\begin{center}
\includegraphics[width=6in]{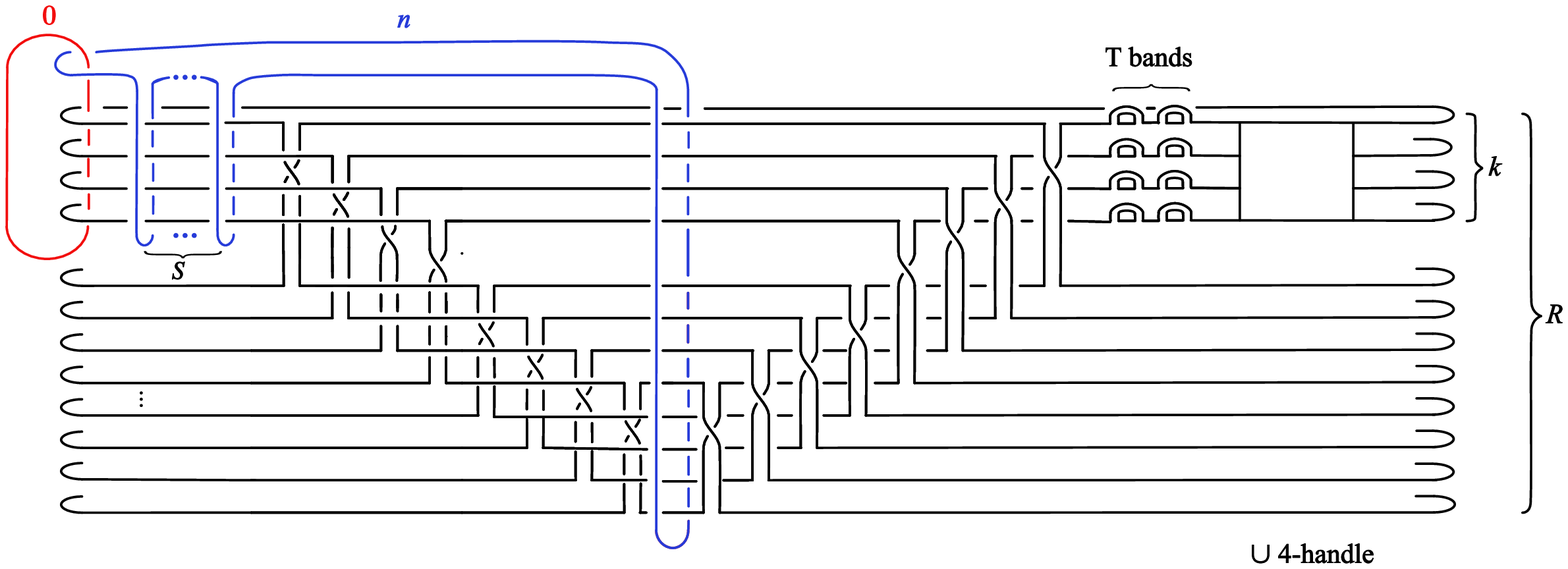}
\caption{The ribbon surface $\Sigma(R,S,T)$ has $R$ horizontal disks. The $n$-framed attaching circle links the horizontal disks $S$ times positively in the indicated region. There are $T$ trivial bands attached to the top $k$ horizontal disks.}
\label{fig:lemma_0}
\end{center}
\end{figure}

\begin{Lem}[The Isotopy Lemma] \label{lem:recursive} Let $k\geq 2$. For any $R \geq 2k$, the ribbon surface $\Sigma(R,S,T)$ is isotopic to the ribbon surface $\Sigma(R-k, S+1, T+k)$.
\end{Lem}

\begin{proof}
The strategy of the proof is to use the long bands on the left side of Figure~\ref{fig:lemma_0} -- the ones that pass inside the horizontal disks -- to cancel the lower horizontal disks. To do so, we must first move the $n$-framed attaching circle and the outside long bands out of the way.

To execute this strategy, we use a 2-handle band dive to arrive at Figure~\ref{fig:lemma_1}. Note that $S$ changes to $S+1$. 
\begin{figure}[ht]
\begin{center}
\includegraphics[width=6in]{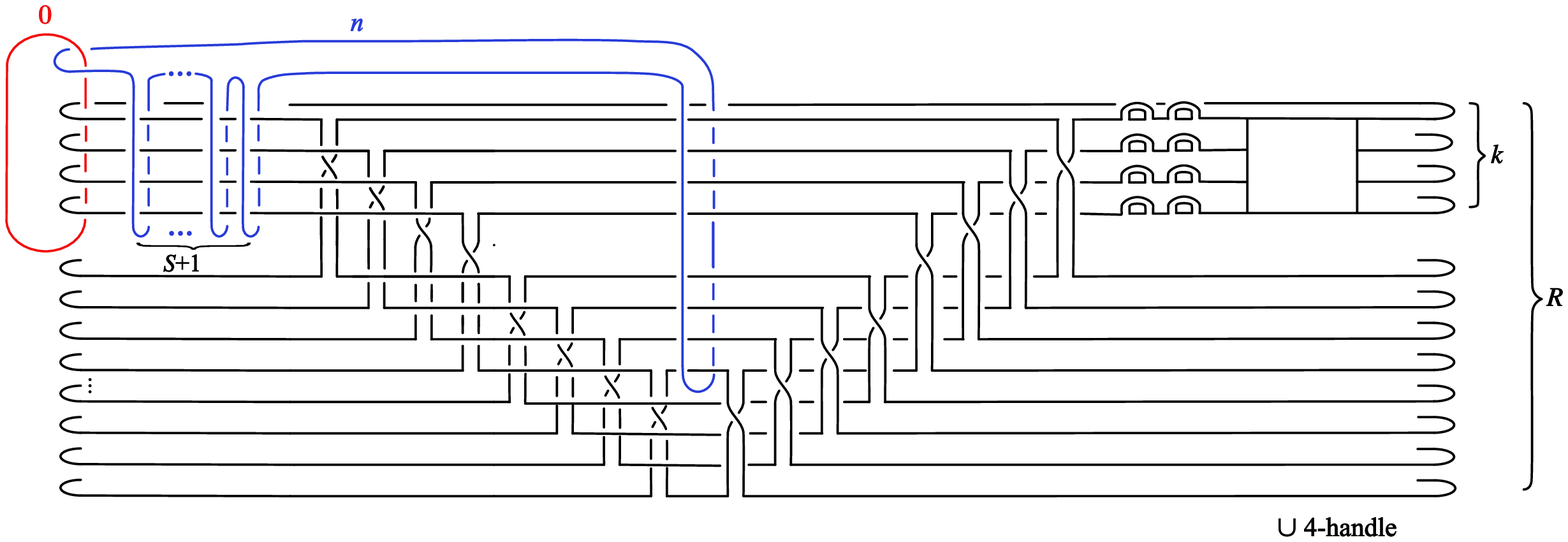}
\caption{}
\label{fig:lemma_1}
\end{center}
\end{figure}
We next use a band slide to obtain Figure~\ref{fig:lemma_2}, followed by a band dive to arrive at Figure~\ref{fig:lemma_3}. 
\begin{figure}[ht]
\begin{center}
\includegraphics[scale=.7]{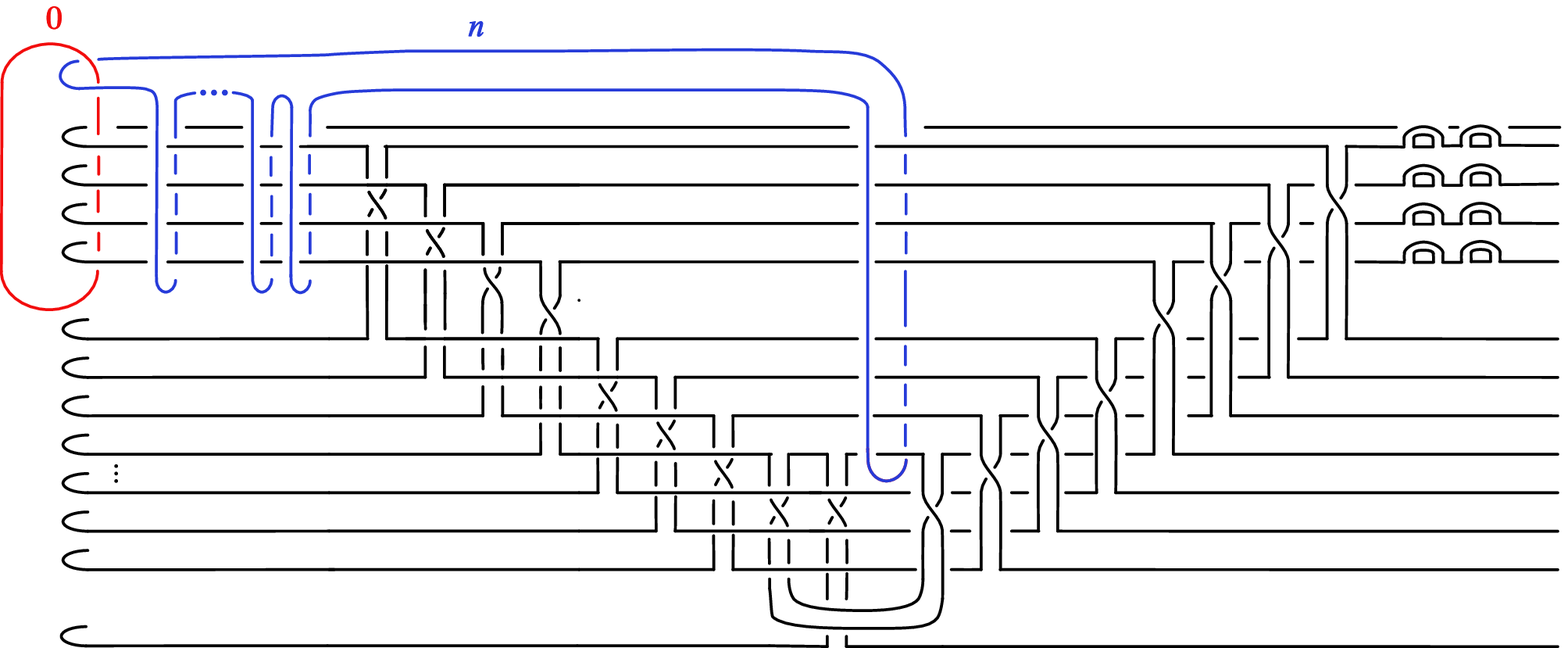}
\caption{}
\label{fig:lemma_2}
\end{center}
\end{figure}

\begin{figure}[ht]
\begin{center}
\includegraphics[scale=.7]{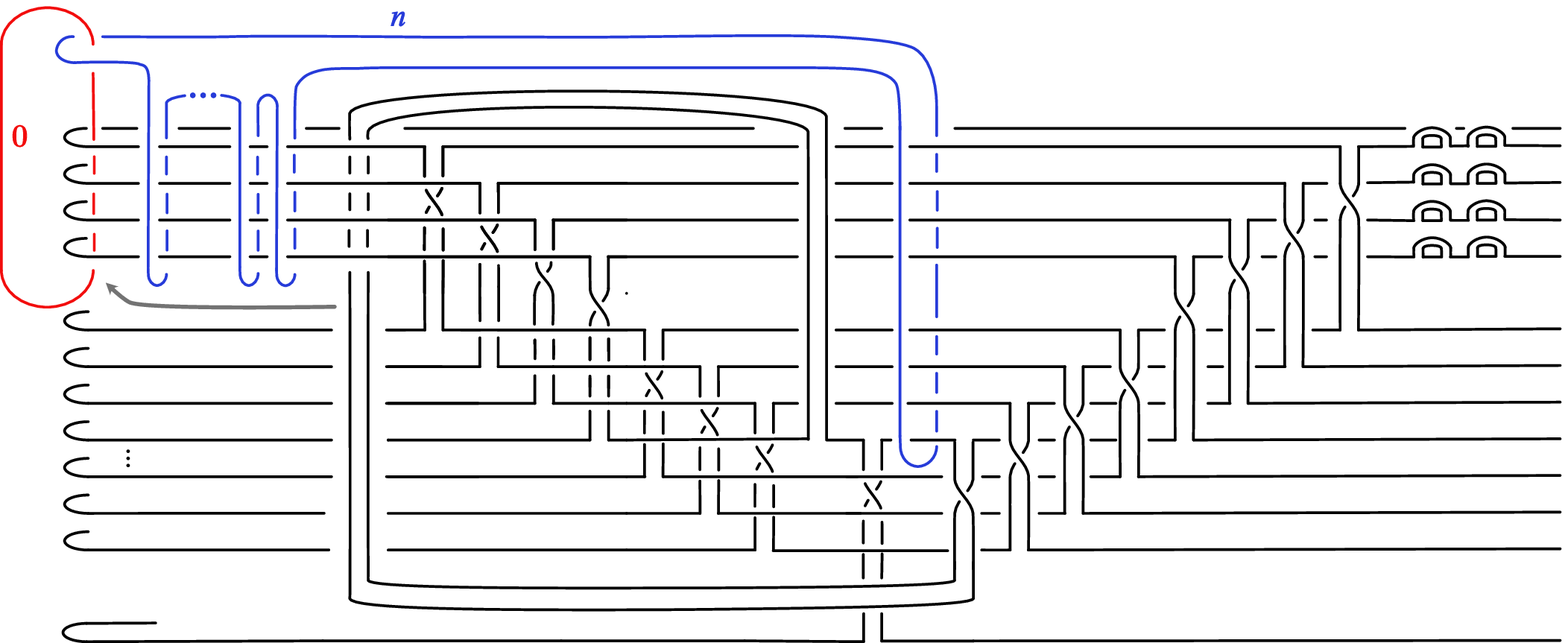}
\caption{}
\label{fig:lemma_3}
\end{center}
\end{figure}
A 2-handle band slide over the 0-framed 2-handle, using the indicated band untangles the band from the upper $k$ horizontal disks, after which it is isotoped to the location shown in Figure~\ref{fig:lemma_4}. This Figure also shows one fewer horizontal disk, as we are now able to cancel it with the remaining attached inside band.
\begin{figure}[ht]
\begin{center}
\includegraphics[scale=.7]{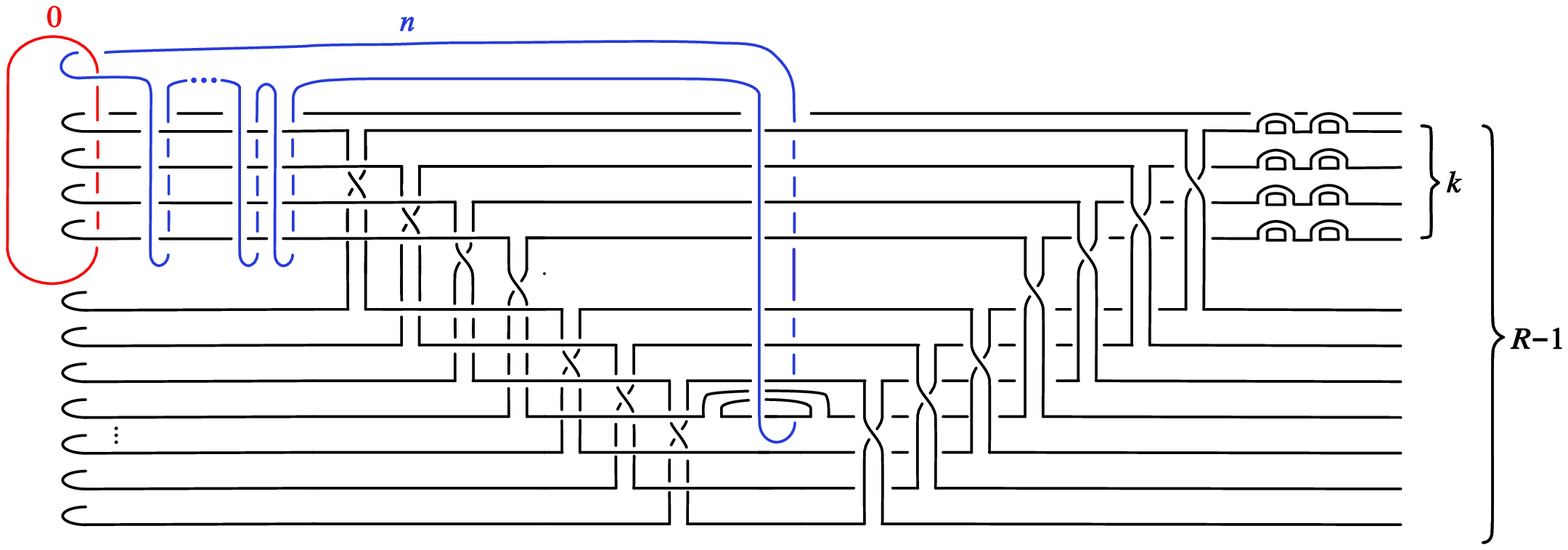}
\caption{}
\label{fig:lemma_4}
\end{center}
\end{figure}

We now repeat the steps between Figures~\ref{fig:lemma_1} and Figure~\ref{fig:lemma_4} $k-1$ times, cancelling $k-1$ more horizontal disks while adding $k-1$ more trivial bands. The outcome is shown in Figure~\ref{fig:lemma_5}. (We have also slid all of the newly introduced trivial bands to the upper $k$ rows of disks.) Altogether, the triple $(R,S,T)$ has been replaced by the triple $(R-k,S+1,T+k)$, proving the Lemma.
\begin{figure}[ht]
\begin{center}
\includegraphics[scale=.7]{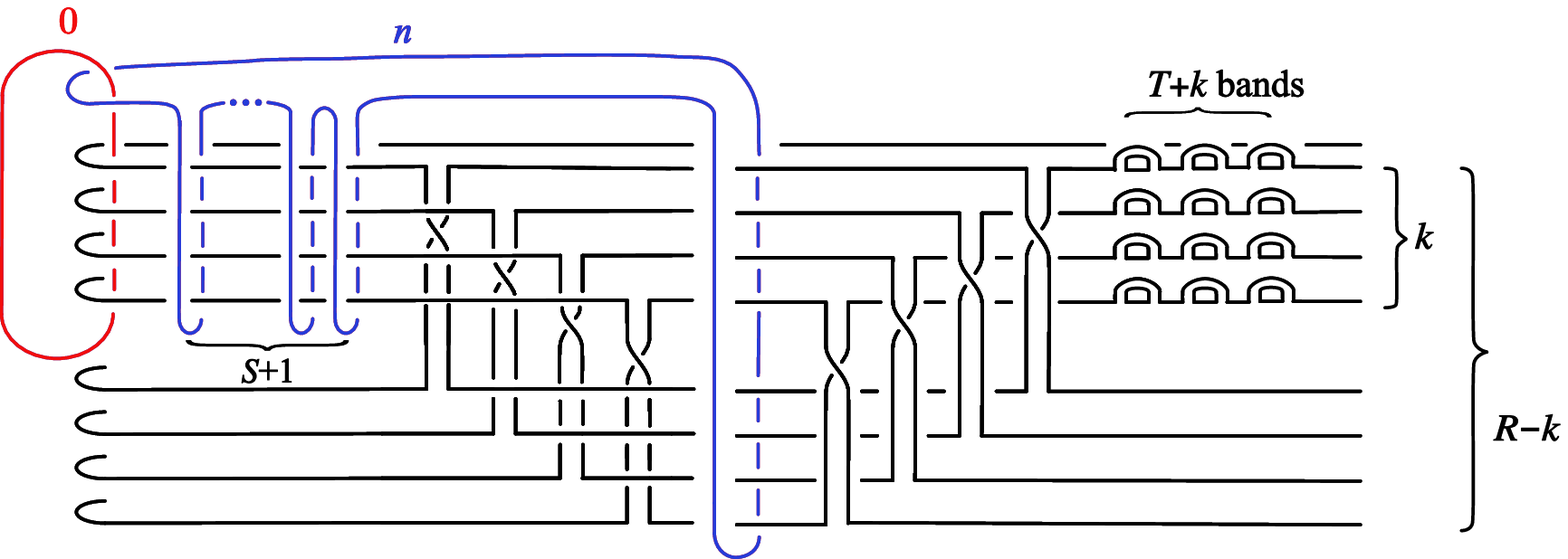}
\caption{}
\label{fig:lemma_5}
\end{center}
\end{figure}
\end{proof}

\subsection{Proof of Theorem \ref{thm:main}} \label{sec:proof}
Observe that in the notation of the Isotopy Lemma of Section~\ref{sec:isotopy_lemma}, the branch surface in Figure~\ref{fig:base_7} is of the form $\Sigma(2g+2,0,0)$, with $k=2h+2$. Applying the Isotopy Lemma $p-1$ times  gives that $X_{g,h}'[i]$ is the 2-fold branched cover of  $\F_{i+1}$ branched over $\Sigma(2h+2+r,p-1,(p-1)(2h+2))$, as seen in Figure~\ref{fig:base_8}. 
\begin{figure}[ht]
\begin{center}
\includegraphics[scale=.8]{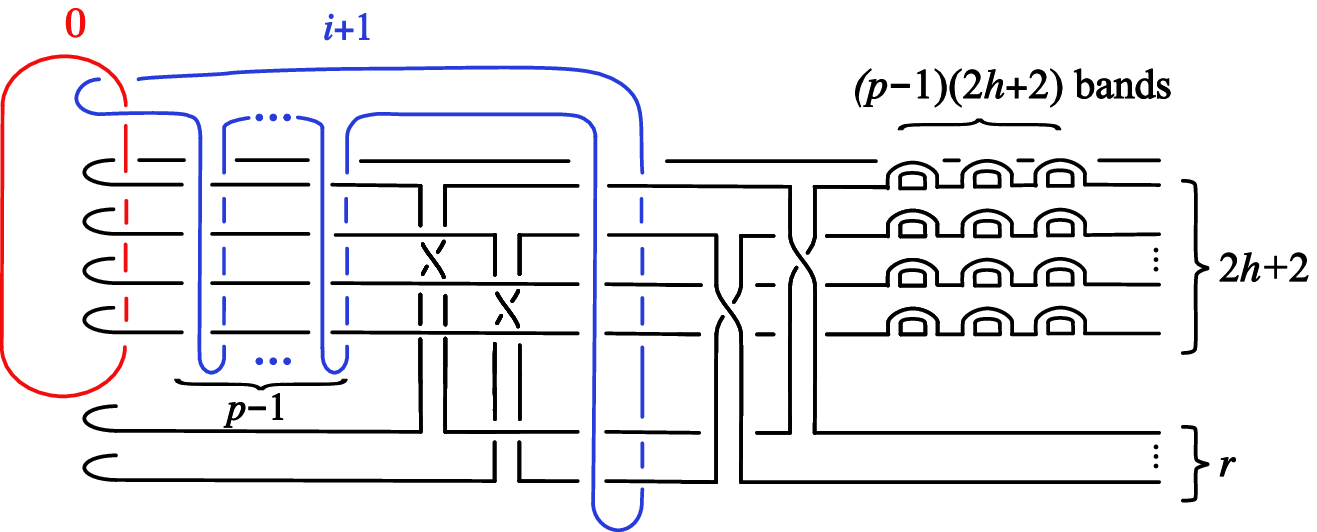}
\caption{}
\label{fig:base_8}
\end{center}
\end{figure}

We continue the process to cancel the bottom $r$ horizontal disks. If $r=0$, these disks and the long bands are not there, and these steps are unnecessary. Otherwise, a 2-handle band dive results in Figure~\ref{fig:base_9}.
\begin{figure}[ht]
\begin{center}
\includegraphics[scale=.8]{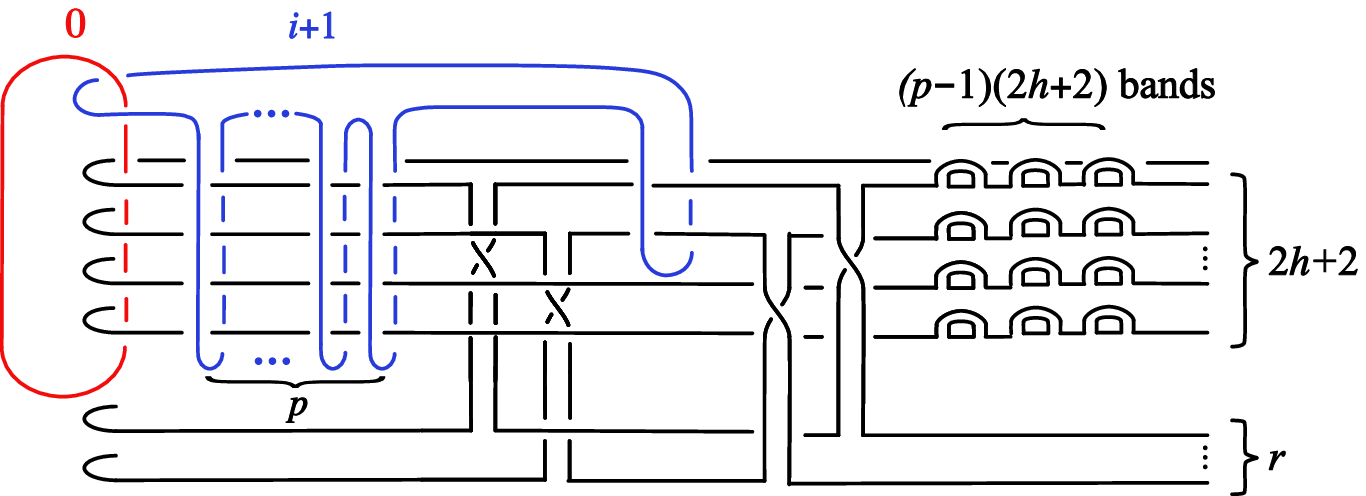}
\caption{}
\label{fig:base_9}
\end{center}
\end{figure}
We then iterate $r$ many times the moves from Figure~\ref{fig:lemma_1} to Figure~\ref{fig:lemma_4} to remove these bottom disks, at the expense of adding $r$ more trivial bands to the branch surface. There are now $(p-1)(2h+2)+r=2g-2h$ trivial bands in total. Cancelling all of them with disks in the 4-handle leaves $2h+2$ remaining in the 4-handle of the result, shown in Figure~\ref{fig:base_10}. Finally we slide the $(i+1)$-framed 2-handle $p$ times over the 0-framed handle, resulting in Figure~\ref{fig:base_11}.
\begin{figure}[ht]
\begin{center}
\includegraphics[scale=.85]{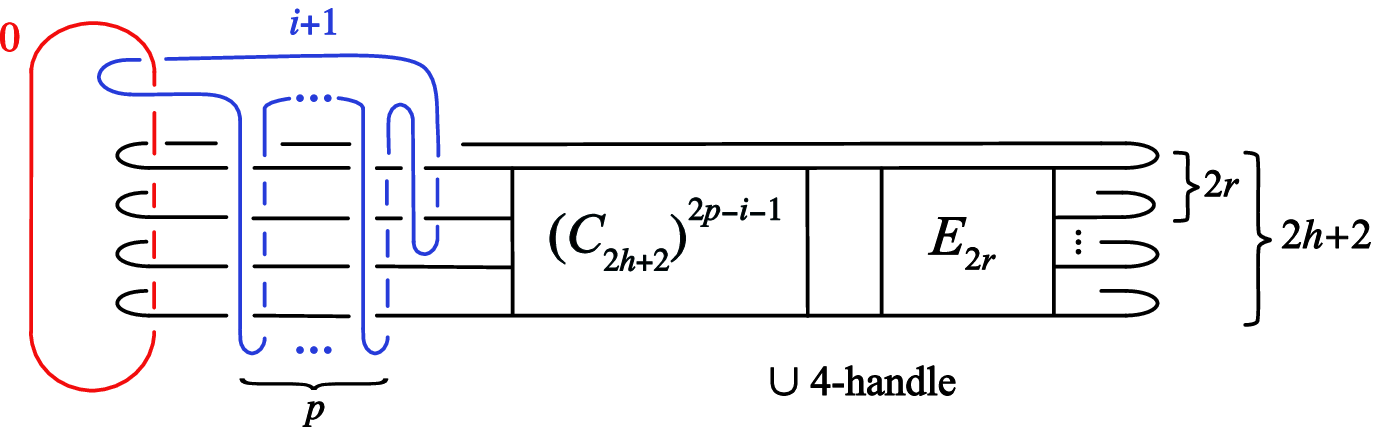}
\caption{}
\label{fig:base_10}
\end{center}
\end{figure}
\begin{figure}[ht]
\begin{center}
\includegraphics[scale=.85]{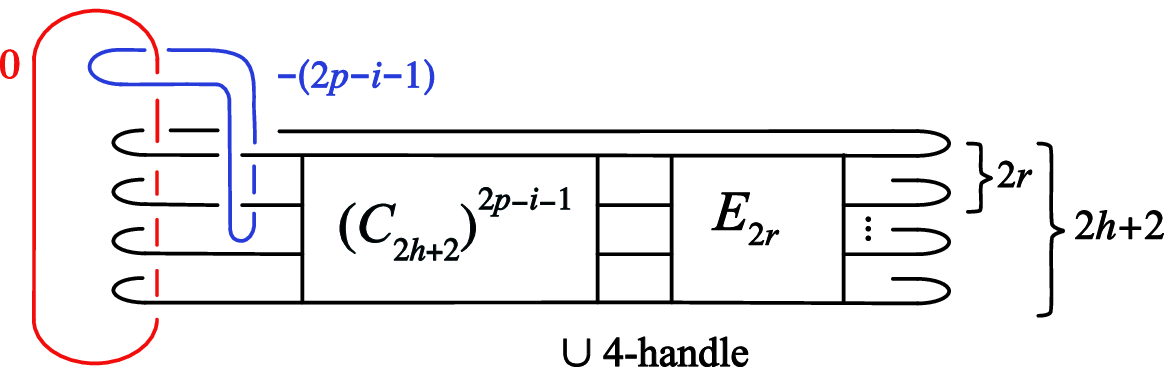}
\caption{The 2-fold branched cover is $Z_h(2p-1-i) \ \#_f \ H_h(r)$.}
\label{fig:base_11}
\end{center}
\end{figure}


Comparing Figure~\ref{fig:base_11}  with the branched cover descriptions of $Z_h$ and $H_h(r)$ in Figures~\ref{fig:Z_h} and \ref{fig:H_h_r} shows that the 2-fold branched cover of Figure~\ref{fig:base_11}  describes the fiber sum $Z_h(2p-1-i) \ \#_f \ H_h(r)$.

Note that when $i=2p-1$, the boxed collection of bands $(C_{2h+2})^{2p-i-1}$ is omitted from the figures, and Figure~\ref{fig:base_11} reduces to Figure~\ref{fig:H_h_r}; the cover is $H_h(r)$, as required. When $r=0$, the collection  $E_{2r}$ is omitted from the figures. Furthermore, it is straightforward to see that when $r=0$, omitting the unecessary steps in the proof will remove the linking of the $-(2p-i-1)$-framed 2-handle in Figure~\ref{fig:base_11} with the ribbon surface. Thus in this case, comparing Figure~\ref{fig:base_11} to Figure~\ref{fig:Z_h} shows that the 2-fold branched cover is $Z_h(2p-i-1)$, as required. 

In the case where {\em both} $i=2g-1$ and $r=0$, the description of $X'_{g,h}[i]$ as a branched cover omits both boxed collections of bands $(C_{2h+2})^{2p-i-1}$ and $E_{2r}$ from Figure~\ref{fig:base_7}. Then Figure~\ref{fig:base_11} reduces to Figure~\ref{fig:base_12}, which describes the cover $\Sigma_h \times S^2 \to S^2 \times S^2$ obtained as the product of the 2-fold cover $\Sigma_h \to \Sigma_{0,2h+2}$ with the identity. This completes the proof of Theorem~\ref{thm:main}.
\begin{figure}[ht]
\begin{center}
\includegraphics[scale=.85]{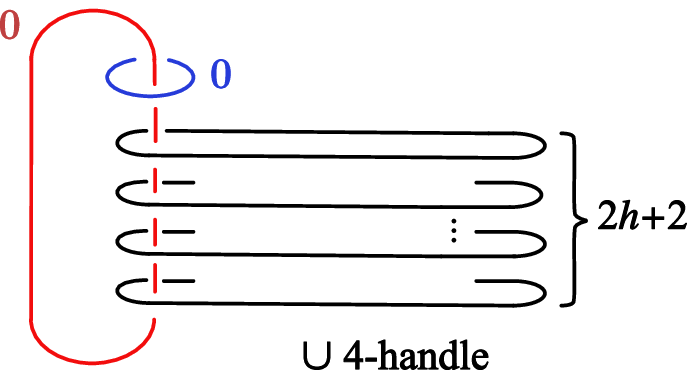}
\caption{The 2-fold branched cover is $\Sigma_h \times S^2$.}
\label{fig:base_12}
\end{center}
\end{figure}

\section{Remarks} \label{sec:remarks}

\subsection{Degree Doubling} \label{subsec:doubling}
Corollary~\ref{cor:infinite_pencils} describes an infinite family of pencils whose fiber genus grows without bound on a fixed symplectic $4$-manifold. This invites comparison with {\em degree doubling}, a construction in which a Donaldson pencil on a symplectic $(X,\omega)$ with fiber class Poincare dual to $d[\omega]$ gives rise to a higher genus pencil on $X$ with fiber class dual to $2d[\omega]$. This operation was studied by Ivan Smith in \cite{Smith_divisors}, and by Denis Auroux and Ludmil Katzarkov for pencils described as branched covers of $\cp$ in \cite{AurouxKatzarkov}. Degree doubling was later shown by Baykur to be an operation applicable to any topological genus $g$ Lefschetz pencil, provided the number of base points does not exceed $2g-2$ (\cite{Baykur_inequivalent}).

It is clear that on a given genus $h$ hyperelliptic Lefschetz fibration, the pencils produced by Corollary~\ref{cor:infinite_pencils} are more abundant than those generated by degree doubling, as their genus grows according to an arithmetic sequence  (with common difference $h+1$),  whereas doubling increases genus exponentially. One may still wonder if pencils generated by doubling occur as a subsequence of our construction. A simple observation shows that this is not the case. In degree doubling, a Lefschetz pencil of genus $g$ with $b$ base points is doubled to become a pencil of genus $g'=2g+b-1$ with $b'=4b$ base points. In particular, the resulting number of base points is always a multiple of $4$. However, this does not hold for some families of pencils described by Corollary~\ref{cor:infinite_pencils}: for instance, $Z_2(1)$ is diffeomorphic to each pencil in the sequence $X'_{3p-1,2}[2p-2]$ for $p \geq 2$, which have $2(i+1) =2(2p-1)=4p-2$ base points, a number never divisible by $4$.

\subsection{Horikawa surfaces}
The even chain relation $(t_{c_1} t_{c_2}t_{c_3} t_{c_4})^{10}=1$ in $\mod(\Sigma_2)$ corresponds to a genus $2$ Lefschetz fibration on a complex surface $Y_2$. For $\ell \geq 3$ and odd, the fiber sums $H_2(2\ell)$ and $Y_2(\ell)$ are homeomorphic, and it remains a long standing open question to determine whether they are diffeomorphic and/or symplectomorphic. (For $\ell$ even they are not homeomorphic. For $\ell=1$, they are homeomorphic but not diffeomorphic \cite{genus2}.)

In \cite{Auroux_Horikawa}, Auroux derived explicit monodromy factorizations for the canonical Lefschetz pencils on both $H_2(6)$ and $Y_2(3)$. Each was seen to be a genus 17 pencil with 16 base points and 196 nodal fibers, and Auroux shows that in fact they are related by a partial twisting operation (or, equivalently, by a Luttinger surgery). Moreover, the corresponding monodromy groups are isomorphic, in contrast to the monodromy groups of the genus 2 fibrations on these manifolds. 

By Corollary~\ref{cor:infinite_pencils}, we see that $H_2(6)$ is diffeomorphic to $X'_{17,2}[7]$ constructed here, which is also a genus 17 pencil with 16 base points and 196 nodes. The factorization for Auroux's pencil (\cite{Auroux_Horikawa}, Theorem 4.4) appears quite a bit different from that in our Theorem~\ref{thm:manifolds} (with $g=17, h=2, i=7$). It is an interesting question to determine if these genus 17 pencils on $H_2(6)$ are Hurwitz equivalent. Moreover, the infinite family of Lefschetz pencils on $H_2(6)$ derived from Corollary~\ref{cor:infinite_pencils} includes some of lower genus, namely $X'_{8,2}[1], X'_{11,2}[3]$ and $X'_{14,2}[5]$. These ``smaller'' pencils may prove useful in studying the differential topology of $H_2(6)$.

Finally, we speculate that there may be a result similar to Theorem~\ref{thm:main} which constructs families of Lefschetz pencils on manifolds defined from the hyperelliptic Lefschetz fibration corresponding to the even chain relation in $\mod(\Sigma_h)$. Such a result would provide pencils which could be compared to Auroux's canonical pencil on $Y_2(3)$ (\cite{Auroux_Horikawa}, Theorem 3.2), and may also supply new smaller genus pencils on that manifold.

\end{document}